\documentclass[11pt,reqno]{amsart}

\usepackage[T1]{fontenc}
\usepackage[utf8]{inputenc}
\usepackage{amsmath,amssymb,mathtools,mathrsfs}
\usepackage{geometry}
\usepackage{enumitem}
\usepackage{microtype}
\usepackage[colorlinks=true,linkcolor=blue,citecolor=blue,urlcolor=blue]{hyperref}
\usepackage[backend=biber,style=numeric,sorting=nyt,sortcites=true,giveninits=true,maxbibnames=99,doi=true,isbn=false,url=false,eprint=true]{biblatex}
\allowdisplaybreaks
\numberwithin{equation}{section}

\newtheorem{theorem}{Theorem}[section]
\newtheorem{proposition}[theorem]{Proposition}
\newtheorem{lemma}[theorem]{Lemma}
\newtheorem{corollary}[theorem]{Corollary}
\theoremstyle{definition}
\newtheorem{definition}[theorem]{Definition}
\theoremstyle{remark}

\newcommand{\R}{\mathbb R}
\newcommand{\N}{\mathbb N}

\newcommand{\E}{\mathbb E}
\newcommand{\Pp}{\mathbb P}

\newcommand{\cH}{\mathfrak H}

\newcommand{\ip}[2]{\left\langle #1,#2\right\rangle}
\newcommand{\norm}[1]{\left\lVert #1\right\rVert}
\newcommand{\abs}[1]{\left\lvert #1\right\rvert}
\newcommand{\dd}{\,\mathrm d}

\title[Optimal Sobolev Rate for Density Approximation]{Optimal Sobolev Rate for Gaussian Density Approximation of Wiener Chaos Vectors}

\author{Huiping Chen}
\address{Huiping Chen, Academy of Mathematics and Systems Science, Chinese Academy of Sciences, Beijing 100190, China}
\email{chenhp@amss.ac.cn}

\author{Yong Chen}
\address{Yong Chen, School of Big Data, Baoshan University, Baoshan, Yunnan 678000, China}
\email{zhishi@pku.org.cn; chenyong77@gmail.com}

\author{Yong Liu}
\address{Yong Liu, LMAM, School of Mathematical Sciences, Peking University, Beijing, 100871, China}
\email{liuyong@math.pku.edu.cn}


\thanks{H. Chen is supported by the China Postdoctoral Science Foundation under Grant Number 2024M763480. Y. Chen is supported by NSFC (No. 12461029) and the PhD Research Startup Fund of Baoshan University (No. BSKY2540). Y. Liu is supported by NSFC (No. 12231002) and the Center for Statistical Science, PKU. The authors used OpenAI's ChatGPT as an auxiliary tool for checking mathematical arguments, language editing and \LaTeX{} template adaptation. All AI-generated suggestions were subsequently verified by the authors, who take full responsibility for the manuscript.}

\begin{document}

\begin{abstract}
	Let $(F_n)$ be a sequence of random vectors with identity
	covariance matrix whose components belong to the same fixed Wiener chaos,
	and assume that $F_n$ converges in law to a standard Gaussian vector.
	We prove that, for every integer $m\geq0$ and every $p\in[1,\infty]$,
	the optimal rate of convergence of the densities in the Sobolev space
	$W^{m,p}$ is given by the maximum of the absolute third-order cumulants
	and the diagonal fourth-order cumulants. The same quantity also gives the optimal rates in total variation, Kolmogorov and $1$-Wasserstein distances. Our proof first derives, by Gaussian interpolation and Gaussian convolution, a cumulant expansion in the space of tempered distributions without imposing Malliavin nondegeneracy at the endpoint. Finite-order Malliavin density estimates then upgrade this identity to Sobolev spaces. Matching lower bounds follow from a finite-dimensional argument in which parity separates the third-order and fourth-order Gaussian
	corrections, while norm equivalence rules out cancellations among
	mixed cumulants. A superconvergence theorem supplies the required finite negative
	moments of the Malliavin determinant along a sufficiently far tail
	of the approximating sequence, so that no Malliavin nondegeneracy
	assumption is required.
\end{abstract}

\maketitle

\tableofcontents

\section{Introduction}\label{sec:introduction}
	
	\subsection{Main results and proof strategy}
	Let $X=\{X(h):h\in\mathfrak H\}$ be an isonormal Gaussian process over a real separable Hilbert space $\mathfrak H$. A central problem in Gaussian analysis is to determine optimal rates of Gaussian approximation for Wiener chaos vectors, especially in strong modes such as density norms and non-smooth probability metrics. The Malliavin--Stein method is well suited to smooth test functions, for which repeated chain-rule arguments expose the cumulants governing the leading error, whereas fewer iterations are available for less regular tests. In the multidimensional setting this leads, in general, to suboptimal bounds; see \cite{Chen2024}, whose optimal result is formulated for suitable smooth distances. In one dimension, passing to non-smooth objects also requires additional ideas; for instance, \cite{NP2015} uses Lusin's theorem to treat total variation, while \cite{EbinaNourdinPeccati2025} develops a generalized functional framework to make singular test objects, such as Dirac masses and their derivatives, compatible with the Malliavin–Stein	and cumulant expansion machinery.
	
	The present paper takes a different route. We combine Gaussian interpolation and Gaussian smoothing to expand the difference between the law of a Wiener chaos vector and the Gaussian law directly in the space of tempered distributions. Finite-order Malliavin density estimates upgrade this identity to Sobolev spaces. For vectors with identity covariance, components in the same fixed Wiener chaos, and a standard Gaussian limit, we show that the optimal rate of convergence of their densities in Sobolev norms is the maximum of the absolute third-order cumulants and the diagonal fourth-order cumulants, with no additional Malliavin nondegeneracy assumption. The same finite-dimensional detection argument applied to the leading Gaussian correction yields matching lower bounds in total variation, Kolmogorov and $1$-Wasserstein distances. Thus the optimal bounds for these non-smooth metrics arise as a direct consequence of the density-level analysis, and no separate Stein argument for non-smooth test functions is needed.
	
	For $k\in\mathbb N$, write $[d]^k:=\{1,\ldots,d\}^k$, and for an ordered index $\alpha\in[d]^k$, let $\kappa_\alpha(F)$ be the corresponding joint cumulant of a random vector $F$; see Section~\ref{sec:preliminaries} for the definitions. For $F=(F_1,\ldots,F_d)$ with finite fourth moments, set 
	\[
	\Delta_3(F)
	:=
	\max_{\alpha\in[d]^3}|\kappa_\alpha(F)|,
	\qquad
	\Delta_4(F)
	:=
	\max_{1\leq i\leq d}\kappa_{iiii}(F),	\qquad 	M(F):=\max\{\Delta_3(F),\Delta_4(F)\}.
	\] 
	When the components of $F$ are multiple Wiener--It\^o integrals, the diagonal fourth-order cumulants $\kappa_{iiii}(F)$, $1\leq i\leq d$, are nonnegative, hence $\Delta_4(F)\geq0$.
	
	For $m\in\mathbb N_0:=\mathbb N\cup\{0\}$ and $p\in[1,\infty]$, let $W^{m,p}(\mathbb R^d)$ be the usual Sobolev space of $f\in L^p(\mathbb R^d)$ such that $\partial_\alpha f:=\partial_{\alpha_1}\cdots\partial_{\alpha_k}f\in L^p(\mathbb R^d)$ for $\alpha=(\alpha_1,\ldots,\alpha_k)\in[d]^k$, $1\leq k\leq m$, with derivatives understood in the distributional sense. We equip $W^{m,p}(\mathbb R^d)$ with the norm
	\[
	\norm{f}_{W^{m,p}(\mathbb R^d)}:=\sum_{k=0}^m\sum_{\alpha\in[d]^k}\norm{\partial_\alpha f}_{L^p(\mathbb R^d)},
	\]
	with the convention $[d]^0=\{\varnothing\}$ and $\partial_\varnothing f=f$. For $p=\infty$, the $L^\infty$ norm is understood as the essential
	supremum. Let $I_q(f)$ be the $q$-th Wiener--It\^o integral of $f\in\mathfrak H^{\odot q}$, where
	$\mathfrak H^{\odot q}$ is the $q$-th symmetric tensor product of
	$\mathfrak H$. Let $\mathcal N_d(0,I_d)$ be the standard Gaussian law on $\mathbb R^d$, where $I_d$ is the $d\times d$
	identity matrix, and write $\operatorname{Cov}(F)$ and $\xrightarrow{d}$ for the covariance matrix of $F$ and convergence in law. One of our main results is the following.
	
	\begin{theorem}\label{thm:main}
		Fix $d\geq1$ and $q\geq2$. For $n\geq1$, let $F_n=(F_{n,1},\ldots,F_{n,d})$ with $F_{n,i}=I_q(f_{n,i})$, $f_{n,i}\in\mathfrak H^{\odot q}$, $1\leq i\leq d$. Assume $\operatorname{Cov}(F_n)=I_d$ and $F_n\xrightarrow{d}\mathcal N_d(0,I_d)$. Then, for every $m\in\mathbb N_0$ and every $p\in[1,\infty]$, there
		exist constants $0<c\leq C<\infty$ and an integer $n_0$ such that,
		for every $n\geq n_0$, $F_n$ has a density
		$p_{F_n}\in W^{m,p}(\mathbb R^d)$ and
		\[
		cM(F_n)
		\leq
		\norm{p_{F_n}-\phi}_{W^{m,p}(\mathbb R^d)}
		\leq
		CM(F_n),
		\]
		where $\phi(x)=(2\pi)^{-d/2}\exp(-|x|^2/2)$, $x\in\mathbb R^d$, is the standard Gaussian density. The constants $c$, $C$ and $n_0$
		may depend on $d,q,m,p$ and on the sequence $(F_n)$, but not on $n$.
	\end{theorem}

	The density estimates developed here also yield the optimal
	rate in three standard probability metrics defined below. For
	random vectors $F$ and $G$, define the total variation
	distance by
	\[
	d_{\mathrm{TV}}(F,G)
	:=
	\sup_{A\in\mathcal B(\mathbb R^d)}
	\left|
	\Pp(F\in A)-\Pp(G\in A)
	\right|,
	\]
	where $\mathcal B(\mathbb R^d)$ denotes the Borel $\sigma$-field of $\mathbb R^d$, and the Kolmogorov distance by
	\[
	d_{\mathrm {Kol}}(F,G)
	:=
	\sup_{x\in\mathbb R^d}
	\left|
	\Pp(F_1\leq x_1,\ldots,F_d\leq x_d)
	-
	\Pp(G_1\leq x_1,\ldots,G_d\leq x_d)
	\right|.
	\]
	Whenever $F$ and $G$ have finite first moments, define the
	$1$-Wasserstein distance by
	\[
	d_{\mathrm W}(F,G)
	:=
	\sup_{\operatorname{Lip}(h)\leq1}
	\left|
	\E[h(F)]-\E[h(G)]
	\right|,
	\qquad
	\operatorname{Lip}(h)
	:=
	\sup_{x\neq y}
	\frac{|h(x)-h(y)|}{|x-y|},
	\]
	where the supremum is over all real-valued Lipschitz functions on
	$\mathbb R^d$. We use the same notation when one or both arguments are probability laws.
	
	\begin{corollary}\label{cor:probability-metrics}
		Under the assumptions of Theorem~\ref{thm:main}, for every $d_* \in \{d_{\mathrm{TV}},d_{\mathrm{Kol}},d_{\mathrm W}\}$, there exist constants
		$0<c\leq C<\infty$ and an integer $n_1$ such that, for every
		$n\geq n_1$,
		\[
		cM(F_n)
		\leq
	d_{*}(F_n,\mathcal N_d(0,I_d))
		\leq
		CM(F_n).
		\]
	\end{corollary}
	
	The proof of Corollary~\ref{cor:probability-metrics} is given in Section~\ref{sec:non-smooth-metrics}. A useful point is that all three lower bounds follow from the same finite-dimensional detection argument. 
	
	We next describe the proof strategy. First, we use Gaussian interpolation. Let $Z$ be an auxiliary standard Gaussian vector independent of $F$ and set $F_t=\sqrt t\,F+\sqrt{1-t}\,Z$, $0\leq t\leq1$; see, for instance, \cite{NourdinPeccatiViens2014} for the smart-path
	method in Gaussian approximation. Differentiating expectations along
	this path and iterating the Malliavin integration by parts formula
	produce a test-function expansion involving cumulants and iterated
	$\Gamma$-variables; see Proposition~\ref{prop:general-interpolation}. Gaussian smoothing for $t<1$ then turns the expansion into an identity for the difference between the law of $F$ and the standard Gaussian law
	in the space of tempered distributions $\mathscr S'(\mathbb R^d)$; see Theorem~\ref{thm:density-interpolation-distribution}. Importantly, neither the existence of a density for $F$ nor Malliavin nondegeneracy of $F$ is required at this
	stage. 	

	This distributional interpolation identity is the point at which our
	method differs from the Malliavin--Stein approach based on generalized functionals used in \cite{EbinaNourdinPeccati2025} to solve the corresponding optimal rate problem in the scalar setting. More precisely, the
	approach of \cite{EbinaNourdinPeccati2025} represents derivatives of the
	density through compositions of Dirac distributions with Wiener
	functionals, solves the corresponding Stein equation in $\mathscr S'(\mathbb R)$, and develops Edgeworth-type expansions within this generalized functional framework. Their upper bound follows from such an expansion, while the lower bound is obtained by retaining one further cumulant order and exploiting the resulting scalar Gaussian
	correction. Here the distributional density identity is
	obtained directly from Gaussian interpolation, classical Malliavin
	integration by parts and Gaussian convolution. This approach has both technical and structural advantages. At the
	technical level, it provides an alternative route based entirely on
	classical Malliavin calculus, rather than on the additional
	functional analytic machinery. At the structural level, it cleanly
	separates the distributional interpolation identity from the density
	regularity estimates used to upgrade it to Sobolev spaces.

	Second,  we formulate finite-order density estimates needed for a prescribed Sobolev norm. The resulting distributional identity in Section~\ref{sec:interpolation} contains ordinary and weighted
	densities along the interpolation path. The next step is therefore to
	identify conditions under which these densities belong to the required
	Sobolev space. The Malliavin integration-by-parts machinery for proving existence and regularity of densities is standard; see,
	for instance, \cite[Section~2.1]{nualart2006malliavin} and
	\cite[Chapter~7]{nualartnualart2018}. However, such results are typically
	formulated under a strong Malliavin nondegeneracy condition, requiring
	the inverse Malliavin determinant to have moments
	of all orders. In 
	Section~\ref{sec:weighted-density}, we track the regularity and integrability
	requirements throughout the argument and show explicitly that, for a
	fixed target space $W^{m,p}(\mathbb R^d)$, only finitely many positive
	Malliavin--Sobolev norms and one sufficiently high but finite negative
	moment of the Malliavin determinant are needed. No
	infinite-order nondegeneracy condition is therefore required for a fixed
	Sobolev norm. In Section~\ref{sec:proof-main}, the fixed-chaos structure supplies the
	positive Malliavin--Sobolev bounds, while \cite[Theorem~4]{HMP2024} provides the
	corresponding finite negative moments along a sufficiently far tail
	of $(F_n)$. Proposition~\ref{prop:interpolation-covariance} then
	transfers these bounds uniformly to the whole interpolation path.
	Consequently, for every prescribed $(m,p)$,  the density $p_{F_n}$ belongs to $W^{m,p}(\mathbb R^d)$ and the distributional
	identity obtained in Section~\ref{sec:interpolation} becomes an
	identity in $W^{m,p}(\mathbb R^d)$ for all sufficiently large $n$.

	Third, we establish the matching lower bound, which is intrinsically multidimensional. The fourth-order interpolation expansion yields
	\[
	p_{F_n}-\phi
	=
	-\frac16\sum_{\alpha\in[d]^3}
	\kappa_\alpha(F_n)\partial_\alpha\phi
	+\frac1{24}\sum_{\alpha\in[d]^4}
	\kappa_\alpha(F_n)\partial_\alpha\phi
	+R_n
	=:\mathcal E_{F_n}+R_n,
	\]
	with $\norm{R_n}_{W^{m,p}}\leq C(M(F_n)^2+\Delta_4(F_n)^{5/4})$.
	The main additional difficulty in the multidimensional problem, which
	has no counterpart in the scalar argument of
	\cite{EbinaNourdinPeccati2025}, is therefore not the interpolation
	expansion itself, but the possible cancellation among the mixed
	cumulants in its leading Gaussian correction $\mathcal E_{F_n}$.
	We show that the third-order and
	fourth-order Gaussian corrections in $\mathcal E_{F_n}$ have opposite parity and use a finite-dimensional
	norm-equivalence argument to rule out cancellation among the cumulant coefficients; see Lemma~\ref{lem:finite-dimensional-detection}.
	It follows that $\norm{\mathcal E_{F_n}}_{L^p}\geq cM(F_n)$.
	Since $M(F_n)\to0$, the remainder $R_n$ is of smaller order and can be
	absorbed, yielding the lower bound in Theorem~\ref{thm:main}. 

	\subsection{Related literature}
	
	The Fourth Moment Theorem of Nualart and Peccati \cite{NP2005} states that a normalized sequence in a fixed Wiener chaos converges to a standard Gaussian law exactly when its fourth moment converges to $3$. Shortly afterwards, Peccati and Tudor \cite{Peccati2005} obtained the multidimensional counterpart. Quantitative versions based on Malliavin calculus and Stein's method were developed in, among others, \cite{NP2009,NOREDDINE20111008,NP2010,NPRMul2010}.
	
	A natural problem is then to identify the optimal rate of Gaussian
	approximation. In the one-dimensional setting, Nourdin and Peccati
	\cite{NP2009b} obtained exact Berry--Esseen asymptotics under an
	additional joint convergence assumption, thereby identifying an
	optimal rate in Kolmogorov distance. Bierm\'e, Bonami, Nourdin and
	Peccati \cite{Hermine2012} showed that, for smooth distances, the rate
	is governed by the third and fourth cumulants, while Nourdin and
	Peccati \cite{NP2015} later obtained the corresponding optimal estimate
	in total variation. In the multidimensional setting, Campese
	\cite{Campese2013} derived optimal rates and one-term Edgeworth
	expansions under additional asymptotic assumptions.
	
	More recently, Chen \cite{Chen2024} considered vectors whose components
	belong to the same fixed Wiener chaos and obtained that the optimal rate for suitable smooth distances is equivalent to the quantity $M(F_n)$ used here, without additional
	asymptotic assumptions. The proof
	combines Malliavin calculus, the multidimensional Stein equation and
	cumulant estimates for iterated $\Gamma$-variables, while the lower
	bound is obtained by a suitable choice of smooth test functions. Such
	an argument is naturally adapted to smooth distances, but does not
	directly provide Sobolev density estimates or sharp bounds for the
	non-smooth metrics considered here.

	A different line of work concerns the existence, regularity and
	convergence of densities. Shigekawa \cite{Shigekawa1980} established
	absolute continuity criteria for finite-dimensional distributions of
	Wiener functionals using Malliavin calculus, while Kusuoka
	\cite{Kusuoka1983} obtained related results for systems of multiple
	Wiener--It\^o integrals by an algebraic method. These questions were
	further investigated for random vectors on Wiener chaos by Nourdin,
	Nualart and Poly \cite{NNP2013}; see also Nualart and Tudor
	\cite{NT2017} for pairs of multiple integrals of the same order. 
	Concerning convergence of densities, Nourdin and Poly \cite{NourdinPoly2013} proved total variation convergence on finite sums of chaoses, while Nourdin, Nualart and Poly \cite{NNP2013} obtained multidimensional total variation and $L^p$ density convergence.  Entropic versions of the fourth moment phenomenon were developed by Nourdin, Peccati and Swan \cite{NPS2014}.

	Quantitative estimates at the density level typically
	require additional Malliavin nondegeneracy. Hu, Lu and Nualart \cite{HLN2014} derived uniform bounds for densities and their derivatives under suitable negative moment assumptions on the
	Malliavin determinant. In the present fixed-chaos setting, combining
	their estimates with the asymptotic nondegeneracy of
	\cite{HMP2024} yields a square-root	fourth-cumulant bound. Theorem~\ref{thm:main} sharpens this to
	the optimal scale $M(F_n)$. Hu, Nualart, Tindel and Xu \cite{HNTX2015} subsequently applied this approach to the Breuer--Major setting. Under suitable assumptions, they obtained the required negative moments and hence uniform convergence of the densities and of their derivatives.
	As observed in \cite{HMP2024}, this already
	establishes a superconvergence result for properly normalized Hermite
	sums of stationary Gaussian sequences. Related
	nondegeneracy conditions were used by Nourdin and Nualart \cite{NN2016}
	in their fourth moment theorem for Fisher information. Quantitative density comparisons based on Malliavin calculus were developed by Bally and Caramellino \cite{BC2014}. More recently, Dang and Hu \cite{DangHu2025} obtained density convergence results for Markov diffusion chaoses and Pearson	targets via Stein's method.

	Herry, Malicet and Poly \cite{HMP2024} established a general
	superconvergence theorem on Wiener chaoses, showing that Gaussian
	convergence automatically yields negative moments of every
	finite order for the Malliavin determinant along a sufficiently far
	tail of the sequence. Consequently, convergence in law is upgraded
	to convergence of the densities in Sobolev spaces, without any additional nondegeneracy assumption. Their result, however, does not identify the optimal Sobolev rate
	studied here. In the scalar case, Ebina, Nourdin and Peccati \cite{EbinaNourdinPeccati2025} subsequently identified the optimal $W^{m,p}(\mathbb R)$ rate as the maximum of the absolute third and fourth cumulants. Mansanarez, Poly and Swan \cite{MansanarezPolySwan2025} further
	established arbitrary-order Edgeworth expansions for scalar Wiener
	chaos in total variation, and considered an approximation problem different from the Sobolev distance problem for densities studied here.

	We finally comment on the assumption that all components belong to
	the same Wiener chaos. This assumption is essential for the quantitative estimates of the iterated $\Gamma$-variables. For mixed chaos orders, the contraction structure suggests that, depending on the relative sizes of the chaos orders, higher-order cumulants may be needed. We plan to investigate this problem in future work.
	
	The paper is organized as follows.
	Section~\ref{sec:preliminaries} recalls basic facts on Wiener chaoses,
	Malliavin calculus, cumulants and iterated $\Gamma$-variables.
	Section~\ref{sec:interpolation} develops the Gaussian interpolation
	and derives the corresponding distributional identity in $\mathscr S'(\mathbb R^d)$.
	Section~\ref{sec:weighted-density} establishes finite-order weighted
	density formulas and Sobolev estimates, and records their stability
	along the Gaussian interpolation path.
	Finally, Section~\ref{sec:proof-main} verifies the required finite-order Malliavin conditions for the approximation sequence, upgrades these distributional identities to Sobolev spaces and derives the optimal rates.

	\section{Preliminaries}\label{sec:preliminaries}
	
	We recall standard notation and facts from Wiener chaos, Malliavin calculus and cumulants; see \cite{Ito51,nourdin2012normal,nualart2006malliavin,nualartnualart2018} for further details. 
	
	Let $\mathfrak H$ be a real separable Hilbert space with inner product $\left\langle \cdot, \cdot\right\rangle _\mathfrak{H}$ and norm $\|\cdot\|_{\mathfrak{H}}$. Let $X=\{X(h):h\in\mathfrak H\}$ be an isonormal Gaussian process on a complete probability space $(\Omega,\mathcal F,\mathbb P)$. That is, $X$ is a centered Gaussian family such that $\mathbb{E}\left[ X(h)X(g) \right]=\left\langle h,g\right\rangle _\mathfrak{H}$ for any $h,g\in \mathfrak{H}$. For $q\geq0$, let $\mathcal{H}_q(X)$ be the $q$-th Wiener--It\^{o} chaos of $X$, which is the closed subspace of $L^2(\Omega)$ generated by $\left\{ H_q(X(h)): h\in\mathfrak{H}, \|h\|_{\mathfrak{H}}=1\right\} $, where $H_q(x)$ is the Hermite polynomial of degree $q$ defined by 
	\begin{equation}\label{eq:Hermite-polynomial}
		\exp\left\lbrace tx-\frac{1}{2}t^2\right\rbrace =\sum_{q=0}^{\infty}\frac{t^q}{q!}H_q(x), \qquad t\in\mathbb{R}.
	\end{equation}
	Let $\mathfrak{H}^{\otimes q}$ and $\mathfrak{H}^{\odot q}$ denote the $q$-th tensor and symmetric tensor product of $\mathfrak{H}$. For $q \geq 1$, the mapping $I_q\left(h^{\otimes q}\right):=H_q(X(h))$ for $h\in \mathfrak{H}$ with $\left\| h\right\| _{\mathfrak{H}}=1$ extends to a linear isometry from $\mathfrak{H}^{\odot q}$, equipped with the norm $\sqrt{q!}\|\cdot\|_{\mathfrak{H}^{\otimes q}}$, onto $\mathcal{H}_{q}(X)$. For $q=0$, we write $I_0(c)=c$ for $c \in \mathbb{R}$. For $f\in \mathfrak{H}^{\odot q}$, $I_q(f)$ is called the $q$-th Wiener--It\^o integral of $f$ with respect to $X$. Wiener--It\^{o} chaoses of different orders are orthogonal, that is, for $f \in \mathfrak{H}^{\odot p}$ and $g \in \mathfrak{H}^{\odot q}$, where $p, q \geq 1$, 
	\begin{equation*}
		\mathbb{E}\left[ I_p(f) I_q(g)\right] =
		\begin{cases}
			p! \left\langle f,g\right\rangle_{\mathfrak{H}^{\otimes p}} ,  & {p=q,}\\
			0,  & {p\neq q.}
		\end{cases}	
	\end{equation*}
	The Wiener--It\^o chaos decomposition of $L^2(\Omega, \sigma(X), \mathbb P)$ implies that
	\[L^2(\Omega, \sigma(X), \mathbb P) = \bigoplus_{q=0}^{\infty} \mathcal{H}_{q}(X).\]
	
	Let $\left\lbrace\eta_k \right\rbrace_{k\geq1} $ be a complete orthonormal system in $\mathfrak{H}$. Given $f\in\mathfrak{H}^{\odot p}$, $g\in\mathfrak{H}^{\odot q}$, for $r=0,\dots,p\land q:= \min\{p,q\}$, define the $r$-th contraction of $f$ and $g$ by 
	\begin{equation*}
		f\otimes_rg=\sum_{i_1,\ldots,i_r=1}^{\infty}\left\langle f,\eta_{i_1}\otimes\cdots\otimes \eta_{i_r}\right\rangle _{\mathfrak{H}^{\otimes r}}\otimes\left\langle g,\eta_{i_1}\otimes\cdots\otimes \eta_{i_r}\right\rangle _{\mathfrak{H}^{\otimes r}}\in \mathfrak{H}^{\otimes (p+q-2r)}.
	\end{equation*}
	This contraction $f\otimes_rg$ is not necessarily symmetric. Its symmetrization is denoted by $f\tilde{\otimes}_rg$. The product formula \cite[Proposition 2.7.10]{nourdin2012normal} gives, for $f\in\mathfrak{H}^{\odot p}$ and $g\in\mathfrak{H}^{\odot q}$ with $p,q\geq0$, 
	\begin{equation*}
		I_p(f)I_q(g)=\sum_{r=0}^{p\land q}r!\binom{p}{r}\binom{q}{r}I_{p+q-2r}(f\tilde{\otimes}_rg).
	\end{equation*}

	For $k\in\N_0$, let $C_p^k(\mathbb R^d)$ and $C_b^k(\mathbb R^d)$ be the classes of $k$-times continuously differentiable functions whose partial derivatives up to order $k$ have polynomial growth and are bounded, respectively. We use $C_p^\infty(\mathbb R^d)$ and $C_b^\infty(\mathbb R^d)$ for the corresponding classes of all orders, and $C_c^\infty(\mathbb R^d)$ for the smooth functions with compact support.

	Let $\mathcal S$ be the class of smooth random variables $F=f(X(h_1),\ldots,X(h_j))$, where $j\geq1$, $h_1,\ldots,h_j\in\mathfrak H$, and $f\in C_p^\infty(\mathbb R^j)$. For $F\in \mathcal{S}$, its Malliavin derivative is a $\mathfrak{H}$-valued random element defined as
	\begin{equation*}
		DF=\sum_{i=1}^{j}\frac{\partial f}{\partial x_i}\left(X\left(h_1\right),\dots,X\left(h_j\right)\right)h_i.
	\end{equation*}
	For any $p\ge1$, $D$ is a closable unbounded operator from
	$L^p(\Omega)$ to $L^p(\Omega;\mathfrak H)$. For $k\geq2$, one can define the $k$-th derivative $D^kF\in L^p(\Omega;\mathfrak{H}^{\otimes k})$ by iteration. For $k\in\N_0$ and $p\ge1$, the Malliavin--Sobolev space $\mathbb D^{k,p}$ is the closure of
	$\mathcal S$ under the norm $\|\cdot\|_{k,p}$ given by 
	\begin{equation*}
		\|F\|_{k,p}^{p}=\sum_{i=0}^{k}\mathbb{E}\left(\left\|D^iF\right\|^{p}_{\mathfrak{H}^{\otimes i}}\right).
	\end{equation*} 
	We omit the subscript $k$ when $k=0$. For a separable Hilbert space $V$,
	$\mathbb D^{k,p}(V)$ and $\|\cdot\|_{k,p;V}$ denote the corresponding
	$V$-valued Malliavin--Sobolev space and norm. For finite-dimensional
	vectors and matrices we use the Euclidean and Hilbert--Schmidt norms,
	respectively, and suppress $V$ from the notation when unambiguous. Set  $\mathbb{D}^{\infty}=\bigcap_{k\geq0}\bigcap_{p\geq1}\mathbb{D}^{k,p}$. For every $q\geq1$ and $f\in\mathfrak H^{\odot q}$, one has $I_q(f)\in\mathbb D^\infty$. For $1\leq k\leq q$, $D^kI_q(f)$ is an $\mathfrak H^{\otimes k}$-valued multiple integral of order $q-k$, while $D^kI_q(f)=0$ for $k>q$. The derivative operator $D$ satisfies the chain rule. Specifically, if $\varphi\in C_b^1(\mathbb R^d)$ and $F=\left(F_1, \ldots, F_d\right)\in (\mathbb{D}^{1,2})^d$, then $\varphi(F) \in \mathbb{D}^{1,2}$ and
	\begin{equation}\label{eq:chain-rule}
		D \varphi(F)=\sum_{i=1}^d \frac{\partial \varphi}{\partial x_i}(F) D F_i .
	\end{equation}
	The chain rule still holds if $F \in (\mathbb{D}^{\infty})^d$ and $\varphi\in C_p^1(\mathbb R^d) $. For a random vector $F=(F_1,\ldots,F_d)\in(\mathbb D^{1,2})^d$, its Malliavin matrix is $\gamma_F:=(\langle DF_i,DF_j\rangle_{\mathfrak H})_{1\leq i,j\leq d}$.

	 We use $ \operatorname{Dom} (\cdot)$ to denote the domain of a general operator. Let $\delta:\operatorname{Dom}(\delta)\subset L^2(\Omega;\mathfrak H)\to L^2(\Omega)$ be the divergence operator defined as the adjoint of $D$. Thus $u\in L^2(\Omega;\mathfrak{H})$ belongs to $\operatorname{Dom} (\delta)$, if and only if there exists a constant
	$c_u<\infty$ depending only on $u$, such that 
	\begin{equation*}
		\left|\mathbb{E}\left[\langle D F, u\rangle_{\mathfrak{H}}\right]\right| \leq c_{u} \sqrt{\mathbb{E}[F^{2}]},\qquad  \forall F \in \mathbb{D}^{1,2}.
	\end{equation*}
	In particular, if $u \in \operatorname{Dom} (\delta)$, then $\delta(u)$ is characterized by the following duality relationship
	\begin{equation}\label{eq:ibp}
		\mathbb{E}( F\delta(u))=\mathbb{E}\left(\langle D F, u\rangle_{\mathfrak{H}}\right), \qquad \forall F \in \mathbb{D}^{1,2}.
	\end{equation}
	For $k\in\mathbb N_0$ and $p>1$, the divergence operator $\delta$ is continuous from
	$\mathbb D^{k+1,p}(\mathfrak H)$ into $\mathbb D^{k,p}$, with
	\begin{equation}\label{eq:delta-continuity-H-proof}
		\norm{\delta(u)}_{k,p}
		\leq C_{k,p}\norm{u}_{k+1,p;\mathfrak H}.
	\end{equation}
	
	Let $L=-\sum_{p=0}^{\infty}pJ_p$ be the infinitesimal generator of the
	Ornstein--Uhlenbeck semigroup
	$T_t=\sum_{p=0}^{\infty}e^{-pt}J_p$, where $J_p$ is the orthogonal
	projection onto the $p$-th Wiener chaos. By
	\cite[Proposition~1.4.3]{nualart2006malliavin},
	$F\in\operatorname{Dom}(L)$ if and only if
	$F\in\operatorname{Dom}(\delta D)$, and then
	\begin{equation}\label{eq:delta-DL}
		\delta DF=-LF.
	\end{equation}
	The pseudo-inverse
	$L^{-1}=-\sum_{p=1}^{\infty}p^{-1}J_p$ satisfies, for every
	$F\in L^2(\Omega)$,
	$L^{-1}F\in\operatorname{Dom}(L)$ and
	\begin{equation}\label{LL-1}
		LL^{-1}F=F-\mathbb E[F].
	\end{equation}
	Combining \eqref{eq:ibp}, \eqref{eq:delta-DL} and \eqref{LL-1}, we obtain the following useful lemma.

	\begin{lemma}[{\cite[Lemma~3.1]{NP2010b}}]\label{IBP}
		Suppose that $F \in \mathbb{D}^{1,2}$ and $G \in L^2(\Omega)$. Then, $L^{-1} G \in \mathbb{D}^{2,2}$ and 
		\begin{equation*}
			\mathbb{E}[F G]=\mathbb{E}[F] \mathbb{E}[G]+\mathbb{E}\big[\left\langle D F,-D L^{-1} G\right\rangle_{\mathfrak{H}}\big] .
		\end{equation*}
	\end{lemma}

	By Meyer's inequalities \cite[Theorem~1.5.1]{nualart2006malliavin},
	hypercontractivity \cite[Corollary~2.8.14]{nourdin2012normal},
	and the interpolation argument leading to \cite[Equation~(13)]{HMP2024},
	we have the following standard estimate.

	\begin{lemma}
		\label{lem:finite-chaos-lifting}
		Fix $N,r\in\mathbb N_0$ and $p>1$. There exists a constant
		$C=C(N,r,p)<\infty$ such that for every $	Y\in\bigoplus_{j=0}^N\mathcal H_j(X)$, one has
		\begin{equation}\label{eq:finite-chaos-lifting}
			\norm{Y}_{r,p}
			\leq
			C\norm{Y}_{1}.
		\end{equation}
	\end{lemma}

	We recall the ordered-index notation. For $k\geq1$, let $[d]^k:=\{1,\ldots,d\}^k$. We call $\alpha=(\alpha_1,\ldots,\alpha_k)\in[d]^k$ an ordered index of length $|\alpha|=k$. Set $\partial_\alpha:=\partial_{\alpha_1}\cdots\partial_{\alpha_k}$. If $\beta\in[d]^\ell$ with $\ell\geq1$, set $(\alpha, \beta)=(\alpha_1,\ldots,\alpha_k,\beta_1,\ldots,\beta_\ell)\in[d]^{k+\ell}$. We use $\langle\cdot,\cdot\rangle_{\R^d}$ and $|\cdot|$ for the Euclidean inner product and norm. 

	\begin{definition}\label{def:cumulant}
		Let $F=(F_1,\ldots,F_d)$ be a random vector and let $\alpha=(\alpha_1,\ldots,\alpha_k)\in[d]^k$ with $k\geq1$. Assume that $\E[\abs{F}^k]<\infty$. The joint cumulant indexed by $\alpha$ is
		\begin{equation*}
			\kappa_\alpha(F)
			:=\left.(-\mathrm i)^k\partial_\alpha\log\varphi_F(t)\right|_{t=0},
		\end{equation*}
		where $\varphi_F(t):=\E\big[\exp\bigl(\mathrm i\langle t,F\rangle_{\R^d}\bigr)\big],
		\, t\in\R^d$, is the characteristic function of $F$.
	\end{definition}

	We abbreviate $\kappa_{(\alpha_1,\ldots,\alpha_k)}(F)$ as
	$\kappa_{\alpha_1\cdots\alpha_k}(F)$. For example, $\kappa_i(F)=\E[F_i]$ and $\kappa_{ij}(F)=\E[F_iF_j]-\E[F_i]\E[F_j]$ for $ i,j\in[d]$. Cumulants are symmetric: for
	every permutation $\sigma$ of $\{1,\ldots,k\}$,
	\[
	\kappa_{(\alpha_{\sigma(1)},\ldots,\alpha_{\sigma(k)})}(F)
	=
	\kappa_{(\alpha_1,\ldots,\alpha_k)}(F).
	\] 
	
	\begin{definition}\label{def:Gamma}
		Let $F=(F_1,\ldots,F_d)\in(\mathbb D^{1,2})^d$. Set $\Gamma_i(F):=F_i$ for $i\in[d]$.
		Let $k\geq2$ and $\alpha=(\alpha_1,\ldots,\alpha_k)\in[d]^k$. If $\Gamma_{\alpha_1,\ldots,\alpha_{k-1}}(F)$ is a well-defined element of $L^2(\Omega)$, define
		\begin{equation*}
			\Gamma_\alpha(F)=\Gamma_{\alpha_1,\ldots,\alpha_k}(F)
			:=\left\langle
			DF_{\alpha_k},
			-DL^{-1}\Gamma_{\alpha_1,\ldots,\alpha_{k-1}}(F)
			\right\rangle_{\cH}.
		\end{equation*}
	\end{definition}
	By \cite[Lemma 4.3]{NOREDDINE20111008}, if
	$F\in(\mathbb D^\infty)^d$, then $\Gamma_\alpha(F)\in\mathbb D^\infty$ for every ordered index $\alpha$. If $F=\left(I_{q_1}(f_1), \ldots, I_{q_d}(f_d)\right)$, where each $f_i\in\mathfrak{H}^{\odot q_i}$ and $q_i\geq1$, \cite[Theorem 4.6, Equation (4.29)]{NOREDDINE20111008} implies that, for every $k\geq1$, there exists an integer $N=N(k,q_1,\ldots,q_d)$ such that, for every $\alpha\in[d]^k$,
	\begin{equation}\label{eq:Gamma-finite-chaos}
		\Gamma_\alpha(F)\in\bigoplus_{j=0}^{N}\mathcal H_j (X).
	\end{equation}

	The corresponding cumulant--$\Gamma$ relation is recalled next (see \cite[Theorem 4.4]{NOREDDINE20111008}).

	\begin{theorem}\label{thm:cumulant-Gamma}
		Let $k\geq1$ and $\alpha\in[d]^k$. Assume that $F=\left(F_1, \ldots, F_d\right)\in (\mathbb D^{k,2^k})^d$.
		Then we have
		\begin{equation*}
			\kappa_\alpha(F)=(k-1)!\,\E[\Gamma_\alpha(F)].
		\end{equation*}
	\end{theorem}

	\section{Gaussian interpolation and density identities}
	\label{sec:interpolation}

	Throughout this section, let $F=(F_1,\ldots,F_d)\in(\mathbb D^\infty)^d$.
	We use Gaussian interpolation to obtain an identity for the
	difference between the law of $F$ and the standard Gaussian law. The
	identity is first established in the sense of tempered distributions,
	without any Malliavin nondegeneracy assumption on $F$. 

	Let $\mathfrak H_0$ be a closed subspace of $\mathfrak H$ such that $F$ is measurable with respect to $\mathcal G_0:=\sigma\{X(h):h\in\mathfrak H_0\}$. If necessary, we enlarge the underlying Hilbert space $\mathfrak{H}$ so that
	$\mathfrak H_0^\perp$ contains an orthonormal family
	$h_1,\ldots,h_d$. We continue to denote the enlarged Hilbert space, the
	isonormal Gaussian process, and the associated Malliavin operators by
	$\mathfrak H$, $X$, $D$, and $L$, respectively. Set \[Z=(Z_1,\ldots,Z_d):=(X(h_1),\ldots,X(h_d))\sim\mathcal N_d(0,I_d).\] Since $h_1,\ldots,h_d\in\mathfrak H_0^\perp$, the Gaussian vector $Z$ is
	independent of $\mathcal G_0$, and hence of $F$. For $F$ and $Z$ introduced above, we define the Gaussian interpolation
	\[
	F_t:=\sqrt t\,F+\sqrt{1-t}\,Z,
	\qquad 0\le t\le1.
	\]

	We record two consequences of this construction of $Z$. Since $F$ is $\mathcal G_0$-measurable, $DF_i\in L^2(\Omega;\mathfrak H_0)$,
	and therefore
	\begin{equation}\label{eq:F-Z-orthogonal}
		\langle DF_i,h_j\rangle_{\mathfrak H}=0
		\qquad\text{a.s.},
		\qquad i,j\in[d].
	\end{equation}
	Moreover, $L^{-1}\bigl(L^2(\Omega,\mathcal G_0, \mathbb P)\bigr)
	\subset L^2(\Omega,\mathcal G_0, \mathbb P)$, while every $G\in\mathbb D^{1,2}$ that is $\mathcal G_0$-measurable
	satisfies $	DG\in L^2(\Omega;\mathfrak H_0)$.
	It follows recursively that, for every ordered index $\alpha$, $\Gamma_\alpha(F)$ is $\mathcal G_0$-measurable and $-DL^{-1}\Gamma_\alpha(F)\in L^2(\Omega;\mathfrak H_0)$. Consequently,
	\begin{equation}\label{eq:Gamma-Z-orthogonal}
		\left\langle
		h_j,-DL^{-1}\Gamma_\alpha(F)
		\right\rangle_{\mathfrak H}
		=0
		\qquad\text{a.s.},
		\qquad j\in[d].
	\end{equation}

	\subsection{Gaussian smoothing along the interpolation path}
	
	For $0\leq t<1$, Gaussian convolution gives smooth densities of $F_t$ independently of any Malliavin nondegeneracy assumption. This provides a convenient density representation for the distributional interpolation argument.
	
	\begin{proposition}
		\label{prop:gaussian-smoothing}
		For every $0\leq t<1$ and $G\in L^1(\Omega,\mathcal G_0)$, the signed measure $\mu_{F_t,G}(\cdot) :=\E\left[ G\mathbf 1_{\{F_t\in \cdot\}} \right] $ has a smooth density
		\begin{align}\label{eq:convolution-weighted-density}
			p_{t,G}(x)
			=
			\E\!\left[G\phi_{1-t}(x-\sqrt t\,F)\right],
		\end{align}
		where $\phi_s(x):=(2\pi s)^{-d/2}\exp(-|x|^2/(2s))$, $s>0$, is the centered Gaussian density with covariance $sI_d$. In the unweighted case $G=1$, we simply write $p_t:=p_{t,1}$ for the
		density of $F_t$.
	\end{proposition}

	\begin{proof}
		Since $F$ is $\mathcal G_0$-measurable and $Z$ is
		independent of $\mathcal G_0$, conditionally on $\mathcal G_0$, $F_t$ is Gaussian with mean $\sqrt t\,F$ and covariance $(1-t)I_d$. Hence, for every bounded Borel function
		$h$, noting that $G$ is $\mathcal G_0$-measurable and using Fubini's theorem, we know that
		\begin{align*}
			\E[G h(F_t)]
			&=
			\E\left[
			G\E[h(F_t)\mid\mathcal G_0]
			\right]=
			\int_{\mathbb R^d}
			h(x)
			\E\left[
			G\phi_{1-t}(x-\sqrt t\,F)
			\right]\dd x,
		\end{align*}
		which proves \eqref{eq:convolution-weighted-density}. For every ordered index $\beta$ and every fixed $0\leq t<1$, $\|\partial_\beta\phi_{1-t}\|_\infty<\infty$.
		Since $G\in L^1$, the dominated convergence theorem yields
		\[
		\partial_\beta p_{t,G}(x)
		=\E\left[G\partial_\beta\phi_{1-t}(x-\sqrt t\,F)\right].
		\]
		Thus $p_{t,G}\in C^\infty(\R^d)$.
	\end{proof}
	
	\subsection{Density interpolation expansions}
	
	For the remainder of this section, assume $\E[F]=0$ and $\operatorname{Cov}(F)=I_d$. We next derive the interpolation identity with test functions. 

	\begin{proposition}
		\label{prop:general-interpolation}
		Let $M\geq3$ and $g\in C_p^{M+1}(\mathbb R^d)$. Then $t\mapsto\E[g(F_t)]$ is continuously differentiable on $(0,1)$ and
		\begin{equation}\label{eq:general-interpolation}
			\frac{\dd}{\dd t}\E[g(F_t)]
			=
			\sum_{r=3}^{M}
			\frac{t^{(r-2)/2}}{2(r-1)!}
			\sum_{\alpha\in[d]^r}
			\kappa_\alpha(F)\E[\partial_\alpha g(F_t)]
			+
			\frac{t^{(M-1)/2}}2
			\sum_{\alpha\in[d]^{M+1}}
			\E[\Gamma_\alpha(F)\partial_\alpha g(F_t)].
		\end{equation}
	\end{proposition}
	
	\begin{proof}
		Set $\Psi_g(t)=\E[g(F_t)]$. Differentiation under the expectation gives
		\[
		\Psi_g'(t)
		=
		\frac1{2\sqrt t}\sum_i
		\E[F_i\partial_i g(F_t)]
		-
		\frac1{2\sqrt{1-t}}\sum_i
		\E[Z_i\partial_i g(F_t)].
		\]
		Gaussian integration by parts in $Z$ yields
		\begin{align*}
			\E[Z_i\partial_i g(F_t)] = \E[ \E[Z_i\partial_i g(F_t) | F ]] = \sqrt{1-t}\, \E[ \E[\partial_{ii}g(F_t) | F ] ]
			=\sqrt{1-t}\,\E[\partial_{ii}g(F_t)].
		\end{align*}
		On the other hand, Lemma~\ref{IBP}, $\E[F]=0$, chain rule \eqref{eq:chain-rule} and \eqref{eq:Gamma-Z-orthogonal} give
		\begin{align*}
			\E[F_i\partial_i g(F_t)]
			&=  \E[F_i] \E[\partial_i g(F_t)] + \E\left[\ip{D (\partial_i g(F_t))}{-D L^{-1}F_i}_{\mathfrak H}\right]\notag\\
			&= \sum_{j=1}^d \E\left[ \partial_{ij} g(F_t) \ip{\sqrt t D F_j + \sqrt{1-t} D Z_j} {-D L^{-1}F_i}_{\mathfrak H}\right] \notag\\
			&=\sqrt t\sum_{j=1}^d \E[\partial_{ij}g(F_t)\Gamma_{ij}(F)].
		\end{align*}
		Since $\E[\Gamma_{ij}(F)]=\kappa_{ij}(F)=\delta_{ij}$, where $\delta_{ij}$ is the Kronecker symbol,
		\begin{equation}\label{eq:Psi-initial}
			\Psi_g'(t)
			=
			\frac12\sum_{\alpha\in[d]^2}
			\E\!\left[
			\bigl(\Gamma_\alpha(F)-\E\Gamma_\alpha(F)\bigr)
			\partial_\alpha g(F_t)
			\right].
		\end{equation}
		For $r\geq2$, set
		\[
		R_r(t)
		:=
		\frac{t^{(r-2)/2}}2
		\sum_{\alpha\in[d]^r}
		\E\!\left[
		\bigl(\Gamma_\alpha(F)-\E\Gamma_\alpha(F)\bigr)
		\partial_\alpha g(F_t)
		\right].
		\]
		By Lemma~\ref{IBP}, chain rule \eqref{eq:chain-rule} and \eqref{eq:Gamma-Z-orthogonal},
		\begin{align*}
			R_r(t)
			&=  \frac{t^{(r-2)/2}}2
			\sum_{\alpha\in[d]^r}  \E\!\left[   \ip{ D\partial_\alpha g(F_t)} {-D L^{-1}  \Gamma_\alpha(F) }_{\mathfrak H}
			\right]  \\
			&=  \frac{t^{(r-2)/2}}2
			\sum_{\eta\in[d]^{r+1}}  \E\!\left[   \ip{ \sqrt t\,DF_{\eta_{r+1}}+\sqrt{1-t}\,h_{\eta_{r+1}}} {-D L^{-1}  \Gamma_{\eta_1,\ldots,\eta_r}(F) }_{\mathfrak H} \partial_\eta g(F_t)
			\right] \\
			&=
			\frac{t^{(r-1)/2}}2
			\sum_{\eta\in[d]^{r+1}}
			\E[\Gamma_\eta(F)\partial_\eta g(F_t)].
		\end{align*}
		If $2\leq r\leq M-1$, splitting $\Gamma_\eta$ into $\Gamma_\eta= \E[\Gamma_\eta(F)] + \Gamma_\eta-\E[\Gamma_\eta(F)] $ and using $\E[\Gamma_\eta(F)]=\kappa_\eta(F)/r!$ from Theorem~\ref{thm:cumulant-Gamma}, we obtain
		\[
		R_r(t)
		=
		\frac{t^{(r-1)/2}}{2r!}
		\sum_{\eta\in[d]^{r+1}}
		\kappa_\eta(F)\E[\partial_\eta g(F_t)]
		+
		R_{r+1}(t).
		\]
		Starting from \eqref{eq:Psi-initial} and iterating this identity for
		$r=2,\ldots,M-1$, while leaving the last representation of $R_M$ unsplit, gives \eqref{eq:general-interpolation}.
	\end{proof}
	
	The preceding formula is a test-function identity and therefore does not require a density at the endpoint $t=1$. Gaussian smoothing for $t<1$ allows us to rewrite it as a distributional identity. Let $\mathscr S(\mathbb R^d)$ denote the Schwartz space and $\mathscr S'(\mathbb R^d)$ its dual, the space of tempered distributions. Let $\mu_F$ denote the law of $F$ and let $\gamma_d$ denote the standard Gaussian measure on $\mathbb R^d$.
	
	\begin{theorem}
		\label{thm:density-interpolation-distribution}
		Let $M\geq3$. Then in $\mathscr S'(\mathbb R^d)$,
		\begin{equation}\label{eq:density-general-distribution}
			\begin{aligned}
				\mu_F-\gamma_d
				={}&
				\sum_{r=3}^{M}
				\frac{(-1)^r}{2(r-1)!}
				\sum_{\alpha\in[d]^r}
				\kappa_\alpha(F)
				\int_0^1t^{(r-2)/2}\partial_\alpha p_t\,\dd t\\
				&+
				\frac{(-1)^{M+1}}2
				\sum_{\alpha\in[d]^{M+1}}
				\int_0^1t^{(M-1)/2}
				\partial_\alpha p_{t,\Gamma_\alpha(F)}\,\dd t .
			\end{aligned}
		\end{equation}
	\end{theorem}
	
	\begin{proof}
		Let $g\in\mathscr S(\mathbb R^d)$. Integrate \eqref{eq:general-interpolation} over $[\varepsilon,1-\varepsilon]$ with $0<\varepsilon<1/2$. 
		Since $F_t\to Z$ as $t\downarrow0$ and $F_t\to F$ as $t\uparrow1$ in $L^q(\Omega)$ for every finite $q$, and since all iterated $\Gamma$-variables have moments of every order, letting $\varepsilon\downarrow0$ and applying the dominated convergence theorem yields
		\[
		\begin{aligned}
			\E[g(F)]-\E[g(Z)]
			={}&
			\sum_{r=3}^{M}
			\frac1{2(r-1)!}
			\sum_{\alpha\in[d]^r}
			\kappa_\alpha(F)
			\int_0^1t^{(r-2)/2}
			\E[\partial_\alpha g(F_t)]\,\dd t\\
			&+
			\frac12
			\sum_{\alpha\in[d]^{M+1}}
			\int_0^1t^{(M-1)/2}
			\E[\Gamma_\alpha(F)\partial_\alpha g(F_t)]\,\dd t.
		\end{aligned}
		\]
		For every $0\le t<1$, by applying Proposition~\ref{prop:gaussian-smoothing} with $G=1$ and $G=\Gamma_\alpha(F)$ and then integrating by parts in $x$, we obtain
		\begin{align*}
			&\E[\partial_\alpha g(F_t)]
			=	\int_{\mathbb R^d}
			\partial_\alpha g(x)p_t(x)\,\dd x=
			(-1)^{|\alpha|}
			\int_{\mathbb R^d}
			g(x) \partial_\alpha p_t(x)\,\dd x,\\
			& \E[
			\Gamma_\alpha(F)\partial_\alpha g(F_t)
			]
			=
			\int_{\mathbb R^d}
			\partial_\alpha g(x)
			p_{t,\Gamma_\alpha(F)}(x)
			\,\dd x =
			(-1)^{|\alpha|} \int_{\mathbb R^d}
			g(x)
			\partial_\alpha
			p_{t,\Gamma_\alpha(F)}(x)
			\,\dd x .
		\end{align*}	
		Since
		\begin{align*}
			\left|
			\E[\partial_\alpha g(F_t)]
			\right|
			\leq
			\norm{\partial_\alpha g}_{L^\infty},\qquad \left|
			\E[
			\Gamma_\alpha(F)\partial_\alpha g(F_t)
			]
			\right|
			\leq
			\E|\Gamma_\alpha(F)|
			\norm{\partial_\alpha g}_{L^\infty}\leq C \norm{\partial_\alpha g}_{L^\infty},
		\end{align*}	
		and all powers of $t$ appearing in
		\eqref{eq:density-general-distribution} are integrable on $(0,1)$, the right-hand side of \eqref{eq:density-general-distribution} defines a continuous linear functional on $\mathscr S(\mathbb R^d)$. This gives \eqref{eq:density-general-distribution} distributionally.
	\end{proof}

	\section{Weighted densities under finite-order regularity}
	\label{sec:weighted-density}

	For a $d$-dimensional random vector $Y$ and a random variable $G\in L^1(\Omega)$, set the finite signed measure $\mu_{Y,G}(A):=\E[G\mathbf 1_{\{Y\in A\}}]$, $A\in\mathcal B(\mathbb R^d)$. Whenever $\mu_{Y,G}$ is absolutely continuous with respect to Lebesgue measure, its density will be denoted by $p_{Y,G}$.

	The interpolation formula of Section~\ref{sec:interpolation} reduces
	the Sobolev approximation problem to the regularity of the ordinary and
	weighted densities along the path. Although Gaussian
	convolution gives smooth densities for every $t<1$, those bounds need
	not remain uniform as $t\uparrow1$. We therefore proceed in two steps.
	First, we give finite-order Malliavin conditions ensuring that the weighted density
	$p_{Y,G}$ belongs to a prescribed Sobolev space $W^{m,p}(\mathbb R^d)$.
	We then return to the Gaussian interpolation and show that these
	conditions, when satisfied at the endpoint, hold uniformly along the
	whole path. 
	
	Let $\mathscr Y$ be a family of $d$-dimensional random vectors with
	components in $\mathbb D^{1,2}$. Let $r\in\mathbb N_0$ and $P,Q>1$. We say that $\mathscr Y$ satisfies the finite-order Malliavin condition $ \mathcal A(r,P,Q)$
	if
	\[
	K_+(r,P):=
	\sup_{Y\in\mathscr Y}\max_{1\leq i\leq d}\norm{Y_i}_{r,P}<\infty,
	\qquad
	K_-(Q):=
	\sup_{Y\in\mathscr Y}\E[(\det\gamma_Y)^{-Q}]<\infty.
	\]
	For a single vector $Y$, we say that $Y$ satisfies $\mathcal A(r,P,Q)$ when the singleton family $\{Y\}$ does. The second condition implies $\det\gamma_Y>0$ almost surely for every $Y\in\mathscr Y$. Constants below are allowed to depend on the corresponding values of $K_+(r,P)$ and $K_-(Q)$.
	
	\subsection{Finite-order estimates for the inverse Malliavin matrix}
	
	The next two results show that a fixed Sobolev norm of $\gamma_Y^{-1}$ requires only one finite positive Malliavin--Sobolev norm of $Y$ and one finite negative moment of $\det\gamma_Y$.
	
	\begin{lemma}\label{lem:gamma-finite-order}
		Fix $r\in\mathbb N_0$ and $p>1$. There exist $P=P(d,r,p)>p$ and $C=C(d,r,p)<\infty$ such that, for every $Y=(Y_1,\ldots,Y_d)\in(\mathbb D^{r+1,P})^d$,
		\begin{equation}\label{eq:gamma-finite-order}
			\norm{\gamma_Y}_{r,p}
			\leq
			C\max_{1\leq i\leq d}\norm{Y_i}_{r+1,P}^2.
		\end{equation}
	\end{lemma}
	
	\begin{proof}
		For $i,j\in[d]$, recall that $ (\gamma_Y)_{ij}=\langle DY_i,DY_j\rangle_{\mathfrak H}$. For every $0\leq k\leq r$, repeated use of the Leibniz rule shows that $D^k(\gamma_Y)_{ij}$ is a finite sum of terms of the form
		\[
		\left\langle D^{\ell+1}Y_i,D^{k-\ell+1}Y_j\right\rangle_{\mathfrak H},
		\qquad 0\leq\ell\leq k.
		\]
		Using H\"older's inequality, with a sufficiently large but finite exponent $P$, we obtain that
		\[
		\norm{D^k(\gamma_Y)_{ij}}_{p}
		\leq
		C\norm{Y_i}_{k+1,P}\norm{Y_j}_{k+1,P}.
		\]
		Summing over $0\leq k\leq r$ and $i,j\in[d]$ proves \eqref{eq:gamma-finite-order}.
	\end{proof}
	
	\begin{lemma}
		\label{lem:inverse-gamma-finite}
		Fix $r\in\mathbb N_0$ and $p>1$. There exist finite exponents $P=P(d,r,p)>1$ and $Q=Q(d,r,p)>1$ such that, if a family $\mathscr Y$ satisfies $\mathcal A(r+1,P,Q)$, then
		\begin{equation}\label{eq:inverse-gamma-conclusion}
			\sup_{Y\in\mathscr Y}\norm{\gamma_Y^{-1}}_{r,p}<\infty.
		\end{equation}
	\end{lemma}
	
	\begin{proof}
		Let $\gamma_Y^*$ denote the adjugate matrix of $\gamma_Y$. Each entry of $\gamma_Y^*$ is a finite sum of products of $d-1$ entries of $\gamma_Y$. Hence, for every fixed $s>1$, applying Lemma~\ref{lem:gamma-finite-order} and H\"older's inequality, we can choose $P,s_1,s_2$ with $s< s_2< s_1< P< \infty$ such that, uniformly in $Y\in\mathscr Y$,
		\begin{align}\label{eq:adjugate-bound}
			\|\gamma_Y^{*}\|_{s} \leq C \|\gamma_Y\|_{s_2}^{d-1} \leq C  \|Y\|_{1,s_1}^{2(d-1)} \leq C \|Y\|_{1,P}^{2(d-1)} <\infty,
		\end{align}
		For $r=0$, since $	\gamma_Y^{-1}=(\det\gamma_Y)^{-1}\gamma_Y^*$, by H\"older's inequality,
		\[
		\norm{\gamma_Y^{-1}}_{p}
		\le \norm{(\det\gamma_Y)^{-1}}_{p_1}
		\norm{\gamma_Y^*}_{p_2},
		\]
		for suitable $1<p_1,p_2<\infty$ satisfying $1/p=1/p_1+1/p_2$. Thus the conclusion follows from \eqref{eq:adjugate-bound} and the assumption $\mathcal A(1,P,Q)$ provided $Q\geq p_1$ and $P$ is chosen sufficiently large. 

		Now let $r\ge1$. Following the regularization argument in
		\cite[Lemma~2.1.6]{nualart2006malliavin}, for $\varepsilon>0$ set
		\[
		\gamma_{Y,\varepsilon}^{-1}
		:=
		\frac{\gamma_Y^*}{\det\gamma_Y+\varepsilon}
		=
		\frac{\det\gamma_Y}{\det\gamma_Y+\varepsilon}\,
		\gamma_Y^{-1}.
		\]
		Applying the Malliavin chain rule to
		$\gamma_{Y,\varepsilon}^{-1}$ and then letting
		$\varepsilon\downarrow0$, using the assumed negative moments of
		$\det\gamma_Y$, justifies the recursive differentiation of $\gamma_Y\gamma_Y^{-1}=I_d$.
		Consequently, every term in
		$D^k\gamma_Y^{-1}$, $1\le k\le r$, is a finite product of the form
		\[
		\gamma_Y^{-1}
		(D^{r_1}\gamma_Y)
		\gamma_Y^{-1}
		\cdots
		\gamma_Y^{-1}(D^{r_\ell}\gamma_Y)
		\gamma_Y^{-1}= (\det\gamma_Y)^{-(\ell+1)}  \gamma_Y^{*} (D^{r_1} \gamma_Y) \gamma_Y^{*} \cdots \gamma_Y^{*} (D^{r_\ell} \gamma_Y) \gamma_Y^{*},
		\]
		where $1\leq\ell\leq k$, $\sum_{j=1}^\ell r_j=k$ and each $r_j\geq1$. Using H\"older's inequality, \eqref{eq:adjugate-bound}, and Lemma~\ref{lem:gamma-finite-order}, we see that all such terms are uniformly bounded in $L^p$ under the condition $\mathcal A(r+1,P,Q)$ provided $P$ and $Q$ are chosen sufficiently large. Since there are only finitely many products for $k\leq r$, both exponents can be chosen finite. This proves \eqref{eq:inverse-gamma-conclusion}.
	\end{proof}

	For $Y=(Y_1,\ldots,Y_d)\in(\mathbb D^{1,2})^d$ with $\det\gamma_Y>0$ a.s., define
	\[
	u_i(Y):=
	\sum_{j=1}^d(\gamma_Y^{-1})_{ij}DY_j,
	\qquad i\in[d].
	\]
	
	\begin{corollary}
		\label{cor:u-finite-order}
		Fix $r\in\mathbb N_0$ and $p>1$. There exist finite exponents $P=P(d,r,p)>1$ and $Q=Q(d,r,p)>1$ such that, if $\mathscr Y$ satisfies $\mathcal A(r+1,P,Q)$, then
		\[
		\sup_{Y\in\mathscr Y}\max_{1\leq i\leq d}
		\norm{u_i(Y)}_{r,p;\mathfrak H}<\infty.
		\]
	\end{corollary}
	
	\begin{proof}
		For $0\leq k\leq r$, the Leibniz rule shows that $D^ku_i(Y)$ is a finite sum of terms of the form $D^\ell(\gamma_Y^{-1})_{ij}\otimes D^{k-\ell+1}Y_j$, $0\leq\ell\leq k$.
		The conclusion follows from H\"older's inequality and Lemma~\ref{lem:inverse-gamma-finite}, after increasing the finite exponents $P$ and $Q$ if necessary.
	\end{proof}
	
	\subsection{Weighted density formulas and Sobolev estimates}
	
	We now iterate the preceding vector fields $u_i(Y)$ with $i\in [d]$. For a sufficiently Malliavin differentiable random variable $G$ and an ordered index $\alpha=(\alpha_1,\ldots,\alpha_\ell)\in[d]^\ell$, define the weights recursively by
	\begin{align}\label{eq:H-recursion}
		H_\varnothing(Y,G):=G,\qquad
		H_\alpha(Y,G)
		:=
		\delta\!\left(
		H_{(\alpha_1,\ldots,\alpha_{\ell-1})}(Y,G)
		u_{\alpha_\ell}(Y)
		\right),
	\end{align}
	whenever the successive divergences are well-defined. The next proposition shows, in particular, that these weights are well-defined under the finite-order assumptions stated there.
	
	\begin{proposition}
		\label{prop:H-finite-order}
		Fix $\ell,r\in\mathbb N_0$ and $p>1$. There exist finite exponents $P=P(d,\ell,r,p)>1$ and $Q=Q(d,\ell,r,p)>1$ such that, if $\mathscr Y$ satisfies $\mathcal A(r+\ell+1,P,Q)$, then there exists $C<\infty$ such that, for every $Y\in\mathscr Y$, every $\alpha\in[d]^\ell$, and every $G\in\mathbb D^{r+\ell,P}$,
		\begin{equation}\label{eq:H-finite-order-estimate}
			\norm{H_\alpha(Y,G)}_{r,p}
			\leq C\norm{G}_{r+\ell,P}.
		\end{equation}
		Moreover, for every $\varphi\in C_b^\ell(\mathbb R^d)$,
		\begin{equation}\label{eq:weighted-ibp}
			\E\bigl[G\partial_\alpha\varphi(Y)\bigr]
			=
			\E\bigl[\varphi(Y)H_\alpha(Y,G)\bigr].
		\end{equation}
	\end{proposition}
	
	\begin{proof}
		We argue by induction on $\ell$. The case $\ell=0$ is immediate since $H_\varnothing(Y,G)=G$. Assume the estimate \eqref{eq:H-finite-order-estimate} has been proved for ordered indices of length $\ell-1$, where $\ell\geq1$. Let $\alpha=(\alpha_1,\ldots,\alpha_\ell)$ and $\widehat\alpha=(\alpha_1,\ldots,\alpha_{\ell-1})$. By the definition of $H_\alpha$ (see \eqref{eq:H-recursion}), the continuity of the divergence operator (see \eqref{eq:delta-continuity-H-proof}) and H\"older's inequality,
		\[
		\norm{H_\alpha(Y,G)}_{r,p} \leq C \norm{
			H_{\hat{\alpha}}(Y,G)
			u_{\alpha_\ell}(Y)
		}_{r+1,p;\mathfrak H} 
		\leq
		C
		\norm{H_{\widehat\alpha}(Y,G)}_{r+1,p_1}
		\norm{u_{\alpha_\ell}(Y)}_{r+1,p_2;\mathfrak H},
		\]
		where $p_1,p_2>p$ are chosen so that $1/p=1/p_1+1/p_2$. Corollary~\ref{cor:u-finite-order} controls the second factor. For the first factor, using the induction hypothesis for length $\ell-1$ with target Sobolev order $r+1$ and integrability exponent $p_1$, we know that there exist finite  $P_1$ and $Q_1$ such that, under $\mathcal{A}(r+\ell+1, P_1, Q_1)$,
		\[
		\norm{H_{\widehat\alpha}(Y,G)}_{r+1,p_1}
		\leq
		C\norm{G}_{(r+1)+(\ell-1),P_1}
		=
		C\norm{G}_{r+\ell,P_1}.
		\]
		Only finitely many positive integrability exponents $P$ and negative-moment orders $Q$ arise in the $\ell$ induction steps. Taking their maxima proves \eqref{eq:H-finite-order-estimate} with finite $P$ and $Q$.
		
		It remains to prove \eqref{eq:weighted-ibp}. From the definition of $u_i(Y)$,
		\[
		\ip{DY_k}{u_i(Y)}_{\cH}  = \sum_{j=1}^d(\gamma_Y^{-1})_{ij} \ip{DY_k}{D Y_j}_{\cH} =   \sum_{j=1}^d(\gamma_Y^{-1})_{ij} (\gamma_Y)_{jk}= \delta_{ik},\qquad i,k\in[d].
		\]
		Consequently, by the chain rule \eqref{eq:chain-rule}, 
		\[
		\left\langle
		D\varphi(Y),u_i(Y)
		\right\rangle_{\mathfrak H}
		=
		\sum_{k=1}^d
		(\partial_k\varphi)(Y)
		\langle DY_k,u_i(Y)\rangle_{\mathfrak H}
		=
		(\partial_i\varphi)(Y).
		\]
		The duality between $D$ and $\delta$ therefore gives
		\begin{align*}
			\E[G (\partial_i \varphi)(Y)]
			=\E\big[\ip{D\varphi(Y)}{Gu_i(Y)}_{\cH}\big]
			=\E\big[\varphi(Y)\delta(Gu_i(Y))\big] = \E\big[\varphi(Y)H_{(i)}(Y,G)\big] .
		\end{align*} 
		Iterating the same argument yields \eqref{eq:weighted-ibp} for every $\alpha\in[d]^\ell$. The estimates \eqref{eq:H-finite-order-estimate} already proved ensure that each divergence operation is well-defined.
	\end{proof}
	
	Polynomial weights will be needed to convert the density formulas in  Proposition~\ref{prop:weighted-density-formula} below into spatial decay estimates in Proposition~\ref{prop:polynomial-density-estimate}.
	
	\begin{corollary}
		\label{cor:H-polynomial-finite}
		Fix $\ell,a\in\mathbb N_0$ and $p>1$. There exist finite exponents $P=P(d,\ell,a,p)>1$ and $Q=Q(d,\ell,a,p)>1$ such that, if $\mathscr Y$ satisfies $\mathcal A(\ell+1,P,Q)$, then there exists $C<\infty$ such that, for every $Y\in\mathscr Y$, $\alpha\in[d]^\ell$, and $G\in\mathbb D^{\ell,P}$,
		\[
		\norm{(1+|Y|)^aH_\alpha(Y,G)}_{p}
		\leq
		C\norm{G}_{\ell,P}.
		\]
	\end{corollary}
	
	\begin{proof}
		Choose $p_1,p_2>1$ with $1/p=1/p_1+1/p_2$. Then by H\"older's inequality,
		\[
		\norm{(1+|Y|)^aH_\alpha(Y,G)}_{p}
		\leq
		\norm{(1+|Y|)^a}_{p_1}
		\norm{H_\alpha(Y,G)}_{p_2}.
		\]
		The first factor is controlled by the positive part  $K_+(\ell+1,P)$ of $\mathcal A(\ell+1,P,Q)$ for a sufficiently large but finite $P$, while the second is controlled by Proposition~\ref{prop:H-finite-order} with $r=0$. Increasing the finite exponents $P,Q$ if necessary proves the result.
	\end{proof}

	We next derive explicit density formulas. The oriented form of the representation is useful because the choice of the orthant can depend on the evaluation point $x$.
	
	\begin{proposition}[Weighted density formula]
		\label{prop:weighted-density-formula}
		Fix $m\in\mathbb N_0$. There exist finite exponents $P=P(d,m)>1$ and $Q=Q(d,m)>1$
		such that the following holds. Let $Y$ satisfy $\mathcal A(d+m+1,P,Q)$ and let $G\in\mathbb D^{d+m,P}$. Then $\mu_{Y,G}$ admits a density $p_{Y,G}\in C^m(\mathbb R^d)$. Moreover, for every $\alpha\in[d]^k$ with $0\leq k\leq m$, and every $A\subset[d]$, with $|A|$ denoting the cardinality of $A$,
		\begin{equation}\label{eq:oriented-density-formula}
			\partial_\alpha p_{Y,G}(x)
			=
			(-1)^{k+|A|}
			\E\!\bigg[
			\bigg( \prod_{i\notin A}\mathbf 1_{\{Y_i>x_i\}}\bigg) 
			\bigg( \prod_{i\in A}\mathbf 1_{\{Y_i<x_i\}}\bigg)
			H_{(1,\ldots,d,\alpha)}(Y,G)
			\bigg].
		\end{equation}
	\end{proposition}
	
	\begin{proof}
		The condition $\mathcal A(d+m+1,P,Q)$ is chosen so that Proposition~\ref{prop:H-finite-order} applies to all weights of length at most $d+m$ appearing below.	Take $h\in C_c^\infty(\mathbb R^d)$ and set
		\[
		\Phi_h(y)
		:=
		\int_{-\infty}^{y_1}\cdots\int_{-\infty}^{y_d}
		h(x)\,\dd x_d\cdots\dd x_1 = \int_{\R^d} h(x)  \prod_{i=1}^d\mathbf 1_{\{y_i>x_i\}} \dd x.
		\]
		Then $\Phi_h\in C_b^\infty(\mathbb R^d)$ and $\partial_1\cdots\partial_d\Phi_h=h$. By \eqref{eq:weighted-ibp} and Fubini's theorem,
		\begin{align*}
			\E[Gh(Y)] &= \E[G(\partial_1\cdots\partial_d\Phi_h)(Y)] =\E[\Phi_h (Y) H_{(1,\dots,d)}(Y,G)]\\
			&=\int_{\R^d}h(x)
			\E\bigg[	\bigg( \prod_{i=1}^d
			\mathbf 1_{\{Y_i>x_i\}} \bigg) H_{(1,\dots,d)}(Y,G)\bigg]\dd x.
		\end{align*}
		This proves absolute continuity of $\mu_{Y,G}$ and \eqref{eq:oriented-density-formula} for $k=0$ and $A=\varnothing$. For $\alpha=(\alpha_1,\dots,\alpha_k)\in[d]^k$ of length $1\leq k\leq m$, apply the same argument to $\partial_\alpha h$. In the distributional sense,
		\[
		\int_{\mathbb R^d}h(x)\partial_\alpha p_{Y,G}(x)\,\dd x
		=
		(-1)^k\E[G\partial_\alpha h(Y)],
		\]
		and a further application of \eqref{eq:weighted-ibp} gives \eqref{eq:oriented-density-formula} with $A=\varnothing$. For a general $A\subset[d]$, replace $\Phi_h$ by
		\[
		\widetilde\Phi_h(y)
		:=
		\int_{\prod_{i\notin A}(-\infty,y_i)
			\times\prod_{i\in A}(y_i,\infty)}
		h(x)\,\dd x.
		\]
		Then $\partial_1\cdots\partial_d\widetilde\Phi_h=(-1)^{|A|}h$. The same argument yields \eqref{eq:oriented-density-formula}.
		
		Finally, the case $G=1$ and $k=0$ shows that the law of $Y$ is absolutely continuous. In particular, $\mathbb P(Y_i=x_i)=0$ for every fixed $i$ and $x_i$. Hence the indicators in \eqref{eq:oriented-density-formula} converge almost surely when $x_n\to x$ in $\R^d$. Using Proposition~\ref{prop:H-finite-order}, we have $H_{\beta}(Y,G)\in L^1$ for any $\beta$ with $0\leq |\beta|\leq d+m $. Then by the dominated convergence theorem, all derivatives of order at most $m$ have continuous versions, and $p_{Y,G}\in C^m(\mathbb R^d)$.
	\end{proof}
	
	The oriented formula \eqref{eq:oriented-density-formula} immediately yields polynomial decay. 
	
	\begin{proposition}
		\label{prop:polynomial-density-estimate}
		Fix $m,a\in\mathbb N_0$. There exist finite exponents $P=P(d,m,a)>1$ and $Q=Q(d,m,a)>1$ such that, if $\mathscr Y$ satisfies $\mathcal A(d+m+1,P,Q)$, then there exists $C<\infty$ such that, for every $Y\in\mathscr Y$ and every $G\in\mathbb D^{d+m,P}$,
		\[
		\max_{0\leq k\leq m}\max_{\alpha\in[d]^k}
		\sup_{x\in\mathbb R^d}
		(1+|x|)^a|\partial_\alpha p_{Y,G}(x)|
		\leq
		C\norm{G}_{d+m,P}.
		\]
	\end{proposition}
	
	\begin{proof}
		Fix $Y,G,k,\alpha$ as in the statement. For $x\in\mathbb R^d$, set $A_x:=\{i\in[d]:x_i<0\}$ and $B_x:=
		\bigcap_{i\notin A_x}\{Y_i>x_i\}
		\cap
		\bigcap_{i\in A_x}\{Y_i<x_i\}$.
		On $B_x$, $|x_i|\leq|Y_i|$ for every $i$, and hence $|x|\leq|Y|$. Therefore by \eqref{eq:oriented-density-formula},
		\[
		(1+|x|)^a|\partial_\alpha p_{Y,G}(x)| = (1+|x|)^a  \left| 	\E\!\left[\mathbf 1_{B_x}
		H_{(1,\ldots,d,\alpha)}(Y,G)\right]\right| 
		\leq
		\E\!\left[
		(1+|Y|)^a
		|H_{(1,\ldots,d,\alpha)}(Y,G)|
		\right].
		\]
		The ordered index $(1,\ldots,d,\alpha)$ has length at most $d+m$. After choosing the finite
		exponents $P$ and $Q$ sufficiently large, depending only on $d,m,a$, Corollary~\ref{cor:H-polynomial-finite}  and H\"older's inequality yield
		\[
		\E\left[
		(1+|Y|)^a
		\left|
		H_{(1,\ldots,d,\alpha)}(Y,G)
		\right|
		\right]
		\leq
		C\norm{G}_{d+m,P}.
		\]
		Taking the supremum over $x\in \R^d$, $0\leq k\leq m$ and $\alpha\in [d]^k$ proves
		the conclusion. 
	\end{proof}
	
	We record the Sobolev estimates that will be used repeatedly later. For a fixed pair $(m,p)$, Corollary~\ref{cor:weighted-Sobolev} below requires only one finite positive exponent $P=P(d,m,p)$ and one finite negative-moment exponent $Q=Q(d,m,p)$. In particular, no condition of the form $(\det\gamma_Y)^{-1}\in\bigcap_{q>1}L^q(\Omega)$ is needed. 
	
	\begin{corollary}[Sobolev estimates for weighted densities]
		\label{cor:weighted-Sobolev}
		Fix $m\in\mathbb N_0$ and $p\in[1,\infty]$. There exist finite exponents $P=P(d,m,p)>1$ and $Q=Q(d,m,p)>1$ such that, if $\mathscr Y$ satisfies $\mathcal A(d+m+1,P,Q)$, then there exists $C<\infty$ such that, for every $Y\in\mathscr Y$ and every $G\in\mathbb D^{d+m,P}$,
		\[
		p_{Y,G}\in W^{m,p}(\mathbb R^d),
		\qquad
		\norm{p_{Y,G}}_{W^{m,p}(\mathbb R^d)}
		\leq
		C\norm{G}_{d+m,P}.
		\]
	\end{corollary}
	
	\begin{proof}
		First let $1\leq p<\infty$ and choose an integer $a$ such that $ap>d$. Proposition~\ref{prop:polynomial-density-estimate} gives that, for every $\alpha\in[d]^k$ with $0\leq k\leq m$,
		\[
		|\partial_\alpha p_{Y,G}(x)|
		\leq
		C\norm{G}_{d+m,P}(1+|x|)^{-a}.
		\]
		Since $ap>d$, $(1+|x|)^{-a}\in L^p(\mathbb R^d)$ and then
		\[
		\norm{\partial_\alpha p_{Y,G}}_{L^p(\mathbb R^d)}
		\leq
		C\norm{G}_{d+m,P}
		\left(
		\int_{\mathbb R^d}
		(1+|x|)^{-ap}\,\dd x
		\right)^{1/p}\leq
		C\norm{G}_{d+m,P}.
		\]
		Summing over $0\leq k\leq m$ proves the result. For $p=\infty$, the derivatives of order at most $m$ are understood as the continuous versions ensured by Proposition~\ref{prop:weighted-density-formula}. Applying Proposition~\ref{prop:polynomial-density-estimate} with $a=0$, we get the estimate directly.
	\end{proof}

	\subsection{Stability along the Gaussian interpolation path}
	We now return to the interpolation of Section~\ref{sec:interpolation}
	and record the stability of the finite-order condition
	$\mathcal A(r,P,Q)$ along the path.
	
	\begin{proposition}
		\label{prop:interpolation-covariance}
		Let $\mathscr F$ be a family of random vectors
		that are measurable with respect to $\mathcal G_0$. If $\mathscr F$ satisfies $\mathcal A(r,P,Q)$, then the interpolation family \[\mathscr F^{\mathrm{int}}:=\{\sqrt t\,F+\sqrt{1-t}\,Z:F\in\mathscr F,\ 0\leq t\leq1\}\] also satisfies $\mathcal A(r,P,Q)$, with possibly different values of the finite bounds $K_+(r,P)$ and $K_-(Q)$.
	\end{proposition}
	
	\begin{proof}
		For every $i\in[d]$, $DF_{t,i}=\sqrt t\,DF_i+\sqrt{1-t}\,DZ_i$. Let $\gamma_t$ denote the Malliavin matrix of $F_t$. Since $DZ_i=h_i$ and \eqref{eq:F-Z-orthogonal},
		we have  
		\begin{align*}
			(\gamma_t)_{ij}
			&=\langle D F_{t,i}, D F_{t,j}\rangle_{\cH}=t\langle DF_i,DF_j\rangle_{\cH}
			+(1-t)\langle DZ_i,DZ_j\rangle_{\cH}=t(\gamma_F)_{ij}+(1-t)\delta_{ij}.
		\end{align*}
		Thus $\gamma_t=t\gamma_F+(1-t)I_d$.

		For two symmetric matrices $A$ and $B$, we consider the Loewner order, that is, we say $A\geq B$ if $A-B$ is a non-negative definite matrix. For any $a\in\R^d$, since $\gamma_F$ and $I_d$ are non-negative definite,
		\begin{align*}
			a^{T} \gamma_t a = t  a^{T} \gamma_F a + (1-t) a^{T}a \geq \begin{cases}
				 a^{T}a/2, & 0\leq t\leq 1/2,\\
				 a^{T} \gamma_F a/2, & 1/2<t\leq 1,
			\end{cases}\Longrightarrow  \gamma_t \geq \begin{cases}
				I_d/2, & 0\leq t\leq 1/2,\\
				\gamma_F/2, & 1/2<t\leq 1.
			\end{cases}
		\end{align*}
		Then $\det\gamma_t \geq 2^{-d}$ for $0\leq t\leq1/2$ and  $ \det\gamma_t\geq 2^{-d}\det\gamma_F$ for  $1/2\leq t\leq1$.
		Hence $(\det\gamma_t)^{-Q}
		\leq
		2^{dQ}
		\bigl(1+(\det\gamma_F)^{-Q}\bigr)$
		uniformly in $t\in[0,1]$.
		Using the assumption that
		$\mathscr F$ satisfies $\mathcal A(r,P,Q)$, we obtain
		\[
		\sup_{F\in\mathscr F}\sup_{0\le t\le1}
		\E\bigl[(\det\gamma_{F_t})^{-Q}\bigr]
		\le
		2^{dQ}\bigl(1+K_-(Q)\bigr)
		<\infty.
		\]
		Finally, since $DF_{t,i}=\sqrt t\,DF_i+\sqrt{1-t}\,h_i$, $D^kF_{t,i}=\sqrt t\,D^kF_i$ for $2\leq k\leq r$, $\mathscr F$ satisfies $\mathcal A(r,P,Q)$ and the Gaussian vector $Z$ has moments of every order, we know that $\mathscr F^{\mathrm{int}}$ also satisfies $\mathcal A(r,P,Q)$.
	\end{proof}

	\section{Proofs of the main results}
	\label{sec:proof-main}

	\subsection{Finite-order regularity for the approximation sequence}\label{subsec:regularity-approximation-sequence}
	
	We apply the framework of Sections~\ref{sec:interpolation} and \ref{sec:weighted-density} to the fixed-chaos sequence of Theorem~\ref{thm:main}.

	Fix $d\geq1$ and $q\geq2$, and let $F=(F_1,\ldots,F_d)$ with $F_i=I_q(f_i)$, $f_i\in\mathfrak H^{\odot q}$, $i\in[d]$. The next proposition provides quantitative bounds on the $\Gamma$-variables appearing in the density expansion \eqref{eq:density-general-distribution}. 
	
	\begin{proposition}
		\label{prop:Gamma-Sobolev-Chen}
		For every $r\in\mathbb N_0$ and $p>1$, there exists
		$C=C(q,d,r,p)<\infty$ such that
		\begin{equation}\label{eq:Gamma-Sobolev-Chen}
			\max_{\alpha\in[d]^4}\norm{\Gamma_\alpha(F)}_{r,p}
			\leq
			C\Delta_4(F),
			\qquad
			\max_{\alpha\in[d]^5}\norm{\Gamma_\alpha(F)}_{r,p}
			\leq
			C\Delta_4(F)^{5/4}.
		\end{equation}
	\end{proposition}
	
	\begin{proof}
		For every $k\geq1$, \eqref{eq:Gamma-finite-chaos} gives an integer
		$N_k=N_k(q)$ such that
		\[
		\Gamma_\alpha(F)
		\in
		\bigoplus_{j=0}^{N_k}\mathcal H_j(X),
		\qquad \alpha\in[d]^k.
		\]
		Hence by Lemma~\ref{lem:finite-chaos-lifting} and \cite[Proposition~3.9]{Chen2024},
		\begin{align*}
		\max_{\alpha\in[d]^k}
		\norm{\Gamma_\alpha(F)}_{r,p}
		\leq
		C
		\max_{\alpha\in[d]^k}
		\E|\Gamma_\alpha(F)|\leq C \begin{cases}
			 \Delta_4(F), & k=4,\\
			 \Delta_4(F)^{5/4}, &k=5.
		\end{cases}
		\end{align*}
		This proves \eqref{eq:Gamma-Sobolev-Chen}. 
	\end{proof}

	We return to the sequence $(F_n)$ of Theorem~\ref{thm:main}.
	The next proposition shows that, for any prescribed exponents
	$R\in\mathbb N_0$ and $P,Q>1$, the corresponding finite-order Malliavin condition
	$\mathcal A(R,P,Q)$ holds uniformly along a sufficiently far tail of the
	sequence.

	\begin{proposition}
		\label{prop:eventual-finite-order-regularity}
		For every $R\in\mathbb N_0$ and $P,Q>1$, there exists an integer
		$ N=N(R,P,Q,d,q,(F_n))$ such that the tail family $\mathscr F_N:=\{F_n:n\geq N\}$ satisfies $\mathcal A(R,P,Q)$.
	\end{proposition}
	
	\begin{proof}
		Since each $F_{n,i}\in \mathcal H_q(X)$ and
		$\norm{F_{n,i}}_2=1$, Lemma~\ref{lem:finite-chaos-lifting} gives
		\[
		\sup_{n\geq1}\max_{1\leq i\leq d}
		\norm{F_{n,i}}_{R,P}<\infty.
		\]
		For the fixed exponent $Q>1$, \cite[Theorem~4]{HMP2024} yields constants $N_Q\geq1$ and $C_Q<\infty$ such that
		\[
		\sup_{n\geq N_Q}
		\E\left[(\det\gamma_{F_n})^{-Q}\right]
		\leq C_Q.
		\]
		Therefore, for $N=N_Q$, the family $\mathscr F_N:=\{F_n:n\geq N\}$ satisfies $\mathcal A(R,P,Q)$.
	\end{proof}
	
	Choose a closed subspace $\mathfrak H_0$ of $\mathfrak H$ such that the whole sequence
	$(F_n)$ is measurable with respect to $\mathcal G_0=
	\sigma\{X(h):h\in\mathfrak H_0\}$.
	Enlarging the underlying Hilbert space if necessary, as in
	Section~\ref{sec:interpolation}, we use the same independent
	standard Gaussian vector $Z$ for all $n$.
	Set
	\begin{equation}\label{eq:Fnt}
		F_{n,t}:=\sqrt t\,F_n+\sqrt{1-t}\,Z,
		\qquad 0\leq t\leq1.
	\end{equation}
	For
	$G\in L^1(\Omega)$, we write $p_{n,t,G}$ for the density of the signed measure $\mu_{F_{n,t},G}(\cdot) =\E\left[ G\mathbf 1_{\{F_{n,t}\in \cdot\}} \right] $ whenever it exists, and put $p_{n,t}=p_{n,t,1}$.
	
	Combining Proposition~\ref{prop:eventual-finite-order-regularity} with
	Proposition~\ref{prop:interpolation-covariance}, the finite-order
	conditions required by Corollary~\ref{cor:weighted-Sobolev} hold
	uniformly along the interpolation paths. We thus obtain the
	following weighted-density estimates.
	
	\begin{proposition}
		\label{prop:uniform-weighted-density-path}
		Fix $s\in\mathbb N_0$ and $p\in[1,\infty]$. Then there exist finite exponents $P=P(d,s,p)>1$ and $Q=Q(d,s,p)>1$, an integer $N$ and a constant $C<\infty$ such that the path family $\mathscr F_N^{\mathrm{int}}
		:=
		\{F_{n,t}:n\geq N,\ 0\leq t\leq1\}$
		satisfies $\mathcal A(d+s+1,P,Q)$ and, for every $n\geq N$, $t\in[0,1]$, and $G\in\mathbb D^{d+s,P}$,
		\begin{equation}\label{eq:uniform-weighted-density-general}
			p_{n,t,G}\in W^{s,p}(\mathbb R^d),
			\qquad \norm{p_{n,t,G}}_{W^{s,p}(\mathbb R^d)}
			\leq
			C\norm{G}_{d+s,P}.
		\end{equation}
		The constants $N,C$ may depend on $d,q,s,p$ and on the sequence $(F_n)$, but not on $n,t$ or $G$.
	\end{proposition}
	
	\begin{proof}
		Choose the finite exponents $P,Q$ furnished by Corollary~\ref{cor:weighted-Sobolev} for $W^{s,p}(\mathbb R^d)$. Proposition~\ref{prop:eventual-finite-order-regularity} gives $N$ such that
		$\mathscr F_N$ satisfies $\mathcal A(d+s+1,P,Q)$. Proposition~\ref{prop:interpolation-covariance} then shows that the whole path family $\mathscr F_N^{\mathrm{int}}$ also satisfies $\mathcal A(d+s+1,P,Q)$. Applying Corollary~\ref{cor:weighted-Sobolev} yields \eqref{eq:uniform-weighted-density-general}.
	\end{proof}

	Applying Proposition~\ref{prop:uniform-weighted-density-path} with $G=1$ and
	$G=\Gamma_\alpha(F_n)$, and using Proposition~\ref{prop:Gamma-Sobolev-Chen} at
	Malliavin--Sobolev order $(d+s, P)$ with corresponding integrability exponent $P$,
	gives the following bounds.

	\begin{corollary}
		\label{cor:uniform-path-Gamma}
		Fix $s\in\mathbb N_0$ and $p\in[1,\infty]$. There exist an integer $N$ and a constant $C<\infty$ such that, for every $n\geq N$,
		\begin{align}
			&\sup_{0\leq t\leq1}
			\norm{p_{n,t}}_{W^{s,p}}
			\leq C,\nonumber\\
			&\max_{\alpha\in[d]^4}
			\sup_{0\leq t\leq1}
			\norm{p_{n,t,\Gamma_\alpha(F_n)}}_{W^{s,p}}
			\leq
			C\Delta_4(F_n),\nonumber \\
			&\max_{\alpha\in[d]^5}
			\sup_{0\leq t\leq1}
			\norm{p_{n,t,\Gamma_\alpha(F_n)}}_{W^{s,p}}
			\leq
			C\Delta_4(F_n)^{5/4}.\label{eq:uniform-Gamma5-density}
		\end{align}
	\end{corollary}

	\subsection{Optimal Sobolev rate}

	We first prove a uniform estimate along the whole interpolation path. Besides yielding the upper bound at $t=1$ in Theorem~\ref{thm:main}, this estimate will be used to control the error after the leading Gaussian correction is extracted.
	
	\begin{proposition}[Sobolev upper bound]
		\label{prop:uniform-path-control}
		Fix $m\in\mathbb N_0$ and $p\in[1,\infty]$. There exist $N\geq1$ and $C<\infty$ such that, for every $n\geq N$,
		\begin{equation}\label{eq:uniform-path-control}
			\sup_{0\leq u\leq1}
			\norm{p_{n,u}-\phi}_{W^{m,p}(\mathbb R^d)}
			\leq
			C M(F_n).
		\end{equation}
	\end{proposition}
	
	\begin{proof}
		By Corollary~\ref{cor:uniform-path-Gamma}, there exist an integer $N$ and a constant $C<\infty$ such that, for every $n\geq N$,
		\[
		\sup_{0\leq t\leq1}
		\norm{p_{n,t}}_{W^{m+3,p}}
		\leq C,\qquad \max_{\eta\in[d]^4}
		\sup_{0\leq t\leq1}
		\norm{p_{n,t,\Gamma_\eta(F_n)}}_{W^{m+4,p}}
		\leq
		C\Delta_4(F_n).
		\]
		Fix $n\geq N$ and $u\in[0,1]$. Integrating the test-function identity
		\eqref{eq:general-interpolation} with $M=3$ over $[0,u]$, and then using Proposition~\ref{prop:uniform-weighted-density-path} together with the uniform
		estimates above, we obtain in $W^{m,p}(\R^d)$,
		\begin{equation*}
			\begin{aligned}
				p_{n,u}-\phi
				={}&
				-\frac14
				\sum_{\alpha\in[d]^3}
				\kappa_\alpha(F_n)
				\int_0^u\sqrt t\,
				\partial_\alpha p_{n,t}\,\dd t+
				\frac12
				\sum_{\eta\in[d]^4}
				\int_0^u t\,
				\partial_\eta
				p_{n,t,\Gamma_\eta(F_n)}\,\dd t .
			\end{aligned}
		\end{equation*}
		The $W^{m,p}$ norm of the first term is bounded by
		\begin{align*}
			\sum_{\alpha\in[d]^3}
			|\kappa_\alpha(F_n)|
			\int_0^u
			\sqrt t\,
			\norm{p_{n,t}}_{W^{m+3,p}}
			\,\dd t\leq
			C\Delta_3(F_n),
		\end{align*}
		while the $W^{m,p}$ norm of the second term is controlled by
		\begin{align*}
			 C \max_{\eta\in[d]^4}
			\sup_{0\leq t\leq1}
			\norm{p_{n,t,\Gamma_\eta(F_n)}}_{W^{m+4,p}} \leq
			C\Delta_4(F_n).
		\end{align*}
		Therefore,
		\[
		\norm{p_{n,u}-\phi}_{W^{m,p}}
		\leq
		C\Delta_3(F_n)+C\Delta_4(F_n)
		\leq
		CM(F_n),
		\]
		uniformly in $u$. For $p=\infty$, the same argument is applied to the bounded continuous representatives ensured by
		Proposition~\ref{prop:weighted-density-formula}. This proves \eqref{eq:uniform-path-control}.
	\end{proof}

	To derive the lower bound, we now keep one more cumulant term in the interpolation formula. This produces a deterministic Gaussian correction and a remainder of smaller order.
	
	For a random vector $F$ with $\E|F|^4<\infty$, define
	\begin{equation}\label{eq:Edgeworth-correction}
		\mathcal E_F
		:=
		-\frac16
		\sum_{\alpha\in[d]^3}
		\kappa_\alpha(F)\partial_\alpha\phi
		+
		\frac1{24}
		\sum_{\eta\in[d]^4}
		\kappa_\eta(F)\partial_\eta\phi.
	\end{equation}
	
	\begin{proposition}
		\label{prop:density-Edgeworth}
		Fix $m\in\mathbb N_0$ and $p\in[1,\infty]$. There exist $N\geq1$ and
		$C<\infty$ such that, for every $n\geq N$,
		\begin{equation}\label{eq:Edgeworth-expand}
			p_{F_n}-\phi
			=
			\mathcal E_{F_n}+R_n,
		\end{equation}
		where
		\begin{equation}\label{eq:Edgeworth-remainder-bound}
			\norm{R_n}_{W^{m,p}(\mathbb R^d)}
			\leq
			C\big(
			M(F_n)^2+\Delta_4(F_n)^{5/4}
			\big).
		\end{equation}
	\end{proposition}
	
	\begin{proof}
		The estimates below use Proposition~\ref{prop:uniform-path-control} at orders $m+3$ and $m+4$, and Corollary~\ref{cor:uniform-path-Gamma} for fifth-order $\Gamma$-weights at order $m+5$. Thus by increasing $N$ if necessary, we assume that all these estimates hold simultaneously. Apply Theorem~\ref{thm:density-interpolation-distribution} with $M=4$. Proposition~\ref{prop:uniform-weighted-density-path} turns the distributional identity into the following identity in $W^{m,p}(\mathbb R^d)$:
		\begin{align*}
			p_{F_n}-\phi
			={}&-\frac14\sum_{\alpha\in[d]^3}\kappa_\alpha(F_n)
			\int_0^1\sqrt t\,\partial_\alpha p_{n,t}\,\dd t
			+\frac1{12}\sum_{\eta\in[d]^4}\kappa_\eta(F_n)
			\int_0^1t\,\partial_\eta p_{n,t}\,\dd t\\
			&-\frac12\sum_{\theta\in[d]^5}
			\int_0^1t^{3/2}\partial_\theta p_{n,t,\Gamma_\theta(F_n)}\,\dd t.
		\end{align*}
		Writing $p_{n,t}=\phi+(p_{n,t}-\phi)$ in the first two terms gives \eqref{eq:Edgeworth-expand}, where the remainder is
		\[
		\begin{aligned}
			R_n
			={}&
			-\frac14
			\sum_{\alpha\in[d]^3}
			\kappa_\alpha(F_n)
			\int_0^1\sqrt t\,
			\partial_\alpha(p_{n,t}-\phi)\,\dd t
			+
			\frac1{12}
			\sum_{\eta\in[d]^4}
			\kappa_\eta(F_n)
			\int_0^1t\,
			\partial_\eta(p_{n,t}-\phi)\,\dd t\\
			&-
			\frac12
			\sum_{\theta\in[d]^5}
			\int_0^1t^{3/2}
			\partial_\theta
			p_{n,t,\Gamma_\theta(F_n)}\,\dd t.
		\end{aligned}
		\]
		For the first term, Proposition~\ref{prop:uniform-path-control} at order $m+3$ gives
		\[
		\sup_{0\leq t\leq1}
		\norm{p_{n,t}-\phi}_{W^{m+3,p}}
		\leq
		CM(F_n).
		\]
		Therefore its $W^{m,p}$ norm is controlled by
		\[
		C\Delta_3(F_n)M(F_n)
		\leq
		CM(F_n)^2.
		\]
		Theorem~\ref{thm:cumulant-Gamma} and Proposition~\ref{prop:Gamma-Sobolev-Chen} give
		\[
		\max_{\eta\in[d]^4}|\kappa_\eta(F_n)|
		=
		6\max_{\eta\in[d]^4}|\E[\Gamma_\eta(F_n)]|
		\le
		6\max_{\eta\in[d]^4}\E|\Gamma_\eta(F_n)|
		\le C\Delta_4(F_n).
		\]
		Then using
		Proposition~\ref{prop:uniform-path-control} at order $m+4$, the $W^{m,p}$ norm of the second term is bounded by
		\[
		C
		\max_{\eta\in[d]^4}|\kappa_\eta(F_n)|
		\sup_{0\leq t\leq1}
		\norm{p_{n,t}-\phi}_{W^{m+4,p}}
		\leq
		C\Delta_4(F_n)M(F_n)
		\leq
		CM(F_n)^2.
		\]
		Finally, \eqref{eq:uniform-Gamma5-density} with $s=m+5$ yields
		\[
		\max_{\theta\in[d]^5}
		\sup_{0\leq t\leq1}
		\norm{p_{n,t,\Gamma_\theta(F_n)}}_{W^{m+5,p}}
		\leq
		C\Delta_4(F_n)^{5/4}.
		\]
		This implies that $C\Delta_4(F_n)^{5/4}$ controls the $W^{m,p}$ norm of the last term. Combining the three estimates proves
		\eqref{eq:Edgeworth-remainder-bound}. 
	\end{proof}
	
	To compare $R_n$ with the principal correction $\mathcal E_{F_n}$ and treat the probability metrics introduced in Section~\ref{sec:introduction}, we isolate the finite-dimensional argument once and for all. Set
	\[
	\mathcal V
	:=
	\operatorname{span}\left\{
	\partial_\alpha\phi:\alpha\in[d]^3\cup[d]^4
	\right\}.
	\]
	
	\begin{lemma}
		\label{lem:finite-dimensional-detection}
		Let $\norm{\cdot}_*$ be any norm on $\mathcal V$. There exists $c_*>0$, depending only on $d$ and the chosen norm, such that, whenever $F$ satisfies $\E|F|^4<\infty$ and
		$\kappa_{iiii}(F)\geq0$ for every $i\in[d]$,
		\begin{equation}\label{eq:finite-dimensional-detection}
			\norm{\mathcal E_F}_*
			\geq
			c_*M(F).
		\end{equation}
	\end{lemma}
	
	\begin{proof}
		For $r=3,4$, let $\mathbb S_{d,r}$ be the finite-dimensional
		space of symmetric $r$-tensors on $\mathbb R^d$, and define
		\[
		\mathcal L(T^{(3)},T^{(4)})
		:=
		\sum_{\alpha\in[d]^3}T^{(3)}_\alpha\partial_\alpha\phi
		+
		\sum_{\eta\in[d]^4}T^{(4)}_\eta\partial_\eta\phi.
		\]
		We first show that $\mathcal L:\mathbb S_{d,3}\times\mathbb S_{d,4}\to\mathcal V$ is injective. Suppose that
		$\mathcal L(T^{(3)},T^{(4)})=0$. Since $\phi$ is even,  $\sum_{\alpha\in[d]^3}T^{(3)}_\alpha\partial_\alpha\phi$
		is odd, whereas $\sum_{\eta\in[d]^4}T^{(4)}_\eta\partial_\eta\phi$
		is even. Evaluating the identity
		$\mathcal L(T^{(3)},T^{(4)})(\cdot)=0$ at $x$ and $-x$, we obtain
		\[
		\sum_{\alpha\in[d]^3}T^{(3)}_\alpha\partial_\alpha\phi=0,
		\qquad
		\sum_{\eta\in[d]^4}T^{(4)}_\eta\partial_\eta\phi=0.
		\]
		For fixed $r\in\{3,4\}$, divide $\sum_{\alpha\in[d]^r}T^{(r)}_\alpha\partial_\alpha\phi=0$ by $\phi$. Since
		$\partial_\alpha\phi/\phi$ is a polynomial whose homogeneous
		part of degree $r$ is
		$(-1)^r x_{\alpha_1}\cdots x_{\alpha_r}$, we obtain
		\[
		P_{T^{(r)}}(x)
		:=
		\sum_{\alpha\in[d]^r}
		T^{(r)}_\alpha x_{\alpha_1}\cdots x_{\alpha_r}
		\equiv0.
		\]
		For every $\alpha\in[d]^r$, since $T^{(r)}$ is symmetric, $0=(\partial_\alpha P_{T^{(r)}})(0)=r!\,T^{(r)}_\alpha$,
		and therefore $T^{(r)}=0$. Then $\mathcal L$ is injective. Since $\|\cdot\|_*$ is a norm on $\mathcal V$,
		\begin{align*}
			(T^{(3)},T^{(4)})
			\longmapsto
			\|\mathcal L(T^{(3)},T^{(4)})\|_*,\qquad (T^{(3)},T^{(4)})
			\longmapsto \max\left\{
			\max_{\alpha\in[d]^3}|T^{(3)}_\alpha|,
			\max_{\eta\in[d]^4}|T^{(4)}_\eta|
			\right\},
		\end{align*}
		are both norms on the finite-dimensional space
		$\mathbb S_{d,3}\times\mathbb S_{d,4}$. 
		Hence, by equivalence of norms, there exists $c_*>0$ such that
		\[
		\|\mathcal L(T^{(3)},T^{(4)})\|_*
		\ge
		c_*\max\left\{
		\max_{\alpha\in[d]^3}|T^{(3)}_\alpha|,
		\max_{\eta\in[d]^4}|T^{(4)}_\eta|
		\right\}.
		\]
		Applying this with
		$T^{(3)}_\alpha=-\kappa_\alpha(F)/6$ and
		$T^{(4)}_\eta=\kappa_\eta(F)/24$, and using
		$\max_{\eta\in[d]^4}|\kappa_\eta(F)|\geq\Delta_4(F)$ under the assumption
		$\kappa_{iiii}(F)\geq0$, proves \eqref{eq:finite-dimensional-detection}.
	\end{proof}
	
	In particular, Lemma~\ref{lem:finite-dimensional-detection} applies with the $L^p(\mathbb R^d)$ norm for every $p\in[1,\infty]$. For later use, we define three functionals on $L^1(\mathbb R^d)$ whose
	restrictions to $\mathcal V$ will be norms. For $h\in L^1(\mathbb R^d)$, set
	\begin{equation}\label{eq:def-norm}
	\begin{aligned}
		\|h\|_{\mathrm{TV}}
		&:=\frac12\norm h_{L^1},\\
		\|h\|_{\mathrm{Kol}}
		&:=\sup_{x\in\mathbb R^d}
		\left|
		\int_{(-\infty,x_1]\times\cdots\times(-\infty,x_d]}
		h(y)\,\dd y
		\right|,\\
		\|h\|_{\mathrm{BL}}
		&:=\sup\left\{
		\left|\int_{\mathbb R^d}g(x)h(x)\,\dd x\right|:
		g\in C_c^\infty(\mathbb R^d),
		\norm g_{L^\infty}\leq1,
		\operatorname{Lip}(g)\leq1
		\right\}.
	\end{aligned}
\end{equation}
	Their restrictions to $\mathcal V$ are norms. This is immediate for $\|\cdot\|_{\mathrm{TV}}$. If $\|h\|_{\mathrm{Kol}}=0$, then for every $x\in \mathbb{R}^d$, the integral of $h$ on $(-\infty,x_1]\times\cdots\times(-\infty,x_d]$ vanishes, and applying $\partial_1\cdots\partial_d$ in the sense of distributions gives $h=0$. If $\|h\|_{\mathrm{BL}}=0$, then rescaling arbitrary functions in $C_c^\infty(\mathbb R^d)$ shows that $\int gh=0$ for every $g\in C_c^\infty(\mathbb R^d)$, hence again $h=0$ as a distribution. 
	
	Recall \cite[Lemma 5.2.4]{nourdin2012normal}, which shows that for $G=I_q(f)$ with $q\geq2$ and $f\in \mathfrak{H}^{\odot q}$, the fourth cumulant $\E[G^4]-3\E[G^2]^2
	\geq0$. Therefore, Lemma~\ref{lem:finite-dimensional-detection} applies, in particular, to the vectors $(F_n)$ considered in Theorem~\ref{thm:main}. Consequently, for every fixed $p\in[1,\infty]$, there exists a constant $c_p>0$ such that
	\begin{equation}\label{eq:Lp-detection}
		\norm{\mathcal E_{F_n}}_{L^p}
		\geq c_pM(F_n),
	\end{equation}
	and there exists a constant $c_0>0$ such that
	\begin{equation}\label{eq:metric-detection}
		\min\left\{
		\|\mathcal E_{F_n}\|_{\mathrm{TV}},
		\|\mathcal E_{F_n}\|_{\mathrm{Kol}},
		\|\mathcal E_{F_n}\|_{\mathrm{BL}}
		\right\}
		\geq c_0M(F_n).
	\end{equation}

	We now complete the proof of the main theorem.

	\begin{proof}[Proof of Theorem~\ref{thm:main}]
		Fix $m\in\mathbb N_0$ and $p\in[1,\infty]$. By
		Proposition~\ref{prop:uniform-weighted-density-path}, after
		discarding finitely many indices $n$, all finite-order
		weighted density estimates required for this fixed pair $(m,p)$
		hold along the interpolation paths. In particular, $F_n$ has a
		density in $W^{m,p}(\mathbb R^d)$ for all sufficiently large $n$.
		
		We first note that $M(F_n)\to0$. Indeed, since each $F_{n,i}\in \mathcal{H}_{q}(X)$ and $\E[F_{n,i}^2]=1$,
		Lemma~\ref{lem:finite-chaos-lifting} yields uniformly bounded
		moments of every fixed order. Hence polynomials of fixed degree in
		the components of $F_n$ are uniformly integrable. Convergence of $F_n$ in law to $\mathcal N_d(0,I_d)$ therefore implies convergence of all fixed joint moments. In particular,
		$\kappa_{ijk}(F_n)\to0$ and
		$\kappa_{iiii}(F_n)=\E[F_{n,i}^4]-3\to0$.
		Since $d$ is fixed, it follows that
		$\Delta_3(F_n)\to0$, $\Delta_4(F_n)\to0$, and hence
		$M(F_n)\to0$.
		
		The upper bound follows from
		Proposition~\ref{prop:uniform-path-control} with $u=1$. For the lower bound,
		Proposition~\ref{prop:density-Edgeworth},
		\eqref{eq:Lp-detection} and
		$0\leq\Delta_4(F_n)\leq M(F_n)$ give
		\begin{align*}
			\norm{p_{F_n}-\phi}_{W^{m,p}}
			&\geq \norm{p_{F_n}-\phi}_{L^p}
			\geq \norm{\mathcal E_{F_n}}_{L^p}
			-\norm{R_n}_{L^p}\\
			&\geq c_pM(F_n)
			-C\bigl(M(F_n)^2+\Delta_4(F_n)^{5/4}\bigr)\\
			&=M(F_n)
			\bigl(c_p-C(M(F_n)+M(F_n)^{1/4})\bigr).
		\end{align*}
		Since $M(F_n)\to0$, for all sufficiently large $n$, $C\bigl(
		M(F_n)+M(F_n)^{1/4}
		\bigr)
		\leq
		c_p/2$.
		Hence, for some sufficiently large $n_0$,
		\begin{align}\label{eq:lower-bound-Wmp}
			\norm{p_{F_n}-\phi}_{W^{m,p}}
			\geq \frac{c_p}{2}M(F_n),
			\qquad n\geq n_0.
		\end{align}
		Combining this with the upper bound completes the proof.
	\end{proof}

	\subsection{Optimal rates for non-smooth probability metrics}
	\label{sec:non-smooth-metrics}
	
	We first derive the polynomially weighted version of the upper estimate
	along the interpolation path. It follows from the same interpolation
	identity as Proposition~\ref{prop:uniform-path-control}, together with
	the polynomial decay estimate of
	Proposition~\ref{prop:polynomial-density-estimate}.
	
	\begin{proposition}\label{prop:polynomial-rate}
		Fix $a\in\mathbb N_0$. There exist $N\geq1$ and $C<\infty$ such that,
		for every $n\geq N$,
		\begin{equation}\label{eq:polynomial-rate}
			\sup_{0\leq u\leq1}\sup_{x\in\mathbb R^d}
			(1+|x|)^a
			|p_{n,u}(x)-\phi(x)|
			\leq
			CM(F_n).
		\end{equation}
	\end{proposition}

	\begin{proof}
		Apply Proposition~\ref{prop:polynomial-density-estimate} with
		density derivative order $m=4$ and weight exponent $a\in\mathbb N_0$. By
		Proposition~\ref{prop:eventual-finite-order-regularity} and
		Proposition~\ref{prop:interpolation-covariance}, there exists $N$ such that the finite-order conditions $\mathcal{A} (d+5, P,Q)$ required in Proposition~\ref{prop:polynomial-density-estimate} hold uniformly for
		the family $\mathscr F_N^{\mathrm{int}}$. Hence
		\[
		\sup_{n\geq N}\sup_{0\leq t\leq1}
		\max_{\alpha\in[d]^3}
		\sup_{x\in\mathbb R^d}
		(1+|x|)^a
		|\partial_\alpha p_{n,t}(x)|
		\leq C.
		\]
		For the weights $\Gamma_\eta(F_n)$, $\eta\in[d]^4$,
		Proposition~\ref{prop:polynomial-density-estimate} together with
		Proposition~\ref{prop:Gamma-Sobolev-Chen} gives, after enlarging the finite exponents and $N$ if necessary, for every $n\geq N$,
		\[
		\sup_{0\leq t\leq1}\max_{\eta\in[d]^4}
		\sup_{x\in\mathbb R^d}
		(1+|x|)^a
		|\partial_\eta p_{n,t,\Gamma_\eta(F_n)}(x)|
		\leq
		C\max_{\eta\in[d]^4}
		\norm{\Gamma_\eta(F_n)}_{d+4,P}
		\leq
		C\Delta_4(F_n).
		\]
		For $u\in[0,1]$, the interpolation identity used in the proof of
		Proposition~\ref{prop:uniform-path-control}, interpreted through the
		continuous representatives furnished by
		Proposition~\ref{prop:weighted-density-formula}, reads
		\[
		p_{n,u}-\phi
		=
		-\frac14
		\sum_{\alpha\in[d]^3}
		\kappa_\alpha(F_n)
		\int_0^u\sqrt t\,\partial_\alpha p_{n,t}\,\dd t
		+
		\frac12
		\sum_{\eta\in[d]^4}
		\int_0^u t\,
		\partial_\eta p_{n,t,\Gamma_\eta(F_n)}\,\dd t.
		\]
		Multiplying by $(1+|x|)^a$, taking the supremum in $x$ and $u$, and
		using the two estimates above yields
		\[
		\sup_{0\leq u\leq1}\sup_{x\in\mathbb R^d}
		(1+|x|)^a|p_{n,u}(x)-\phi(x)|
		\leq
		C\Delta_3(F_n)+C\Delta_4(F_n)
		\leq
		CM(F_n),
		\]
		which proves \eqref{eq:polynomial-rate}.
	\end{proof}
	
	\begin{proof}[Proof of Corollary~\ref{cor:probability-metrics}]
		We first prove the upper bounds. For total variation and Kolmogorov distances, applying Theorem~\ref{thm:main} with $m=0$ and $p=1$ gives
		\[
		d_{\mathrm{Kol}}(F_n,\mathcal N_d(0,I_d))\leq d_{\mathrm{TV}}(F_n,\mathcal N_d(0,I_d))
		=
		\frac12\norm{p_{F_n}-\phi}_{L^1(\mathbb R^d)}\leq CM(F_n).
		\]
		For the $1$-Wasserstein distance, let $h$ be a Lipschitz function with $\operatorname{Lip}(h)\leq1$. Since $|h(x)-h(0)|\leq|x|$ and both densities integrate to one, Proposition~\ref{prop:polynomial-rate} with any integer $a>d+1$ gives
		\begin{align*}
		d_{\mathrm W}(F_n,\mathcal N_d(0,I_d))
		&= 	\sup_{\operatorname{Lip}(h)\leq1} \left|\E[h(F_n)]-\E[h(Z)]\right|
		\leq
		\int_{\mathbb R^d}|x|\,|p_{F_n}(x)-\phi(x)|\,\dd x
		\\
		&\leq
		CM(F_n)
		\int_{\mathbb R^d}|x|(1+|x|)^{-a}\,\dd x
		\leq
		CM(F_n).
		\end{align*}
		For the lower bounds, Proposition~\ref{prop:density-Edgeworth}
		with $m=0$ and $p=1$ gives
		\[
		p_{F_n}-\phi
		=
		\mathcal E_{F_n}+R_n,
		\qquad
		\norm{R_n}_{L^1}
		\leq
		C\big(M(F_n)^2+\Delta_4(F_n)^{5/4}\big).
		\]
		By \eqref{eq:metric-detection} and the definition \eqref{eq:def-norm},
		\begin{align*}
			&\min\left\{
			\|\mathcal E_{F_n}\|_{\mathrm{TV}}, 	\|\mathcal E_{F_n}\|_{\mathrm{Kol}},
			\|\mathcal E_{F_n}\|_{\mathrm{BL}}
			\right\}
			\geq c_0M(F_n),\\
			&\max\left\{  \|R_n\|_{\mathrm{TV}}, 	\| R_n\|_{\mathrm{Kol}},
			\|R_n \|_{\mathrm{BL}} \right\}
			\leq\norm{R_n}_{L^1}.
		\end{align*}
		Finally, note that 
		\begin{align*}
			&d_{\mathrm{TV}}(F_n,\mathcal N_d(0,I_d))
			=
			\| p_{F_n}-\phi\|_{\mathrm{TV}},
			\qquad
			d_{\mathrm{Kol}}(F_n,\mathcal N_d(0,I_d))
			= 	\| p_{F_n}-\phi\|_{\mathrm{Kol}},\\
			&d_{\mathrm W}(F_n,\mathcal N_d(0,I_d))
			\geq
			\| p_{F_n}-\phi\|_{\mathrm{BL}},
		\end{align*}
		where the last inequality follows from the fact that the test class defining $\|\cdot\|_{\mathrm{BL}}$ is contained in the test class defining $d_{\mathrm W}$.
		Therefore, for each
		$d_{*}\in\{d_{\mathrm{TV}},d_{\mathrm{Kol}},d_{\mathrm W}\}$,
		\[
		d_{*}(F_n,\mathcal N_d(0,I_d))
		\geq
		c_0M(F_n)
		-
		C\big(M(F_n)^2+\Delta_4(F_n)^{5/4}\big).
		\]
		Since $0\leq\Delta_4(F_n)\leq M(F_n)$ and $M(F_n)\to0$, as shown in the proof of
		Theorem~\ref{thm:main}, the same argument used to obtain \eqref{eq:lower-bound-Wmp} gives all three lower bounds and completes the proof.
	\end{proof}

\printbibliography

\end{document}